\documentclass[11pt]{amsart}

\usepackage[margin=1.25in]{geometry}

\usepackage{amsmath, amssymb, amsthm, graphicx, enumerate, tikz, float, color}
\usepackage{mathtools, comment}
\usepackage{wasysym}
\usepackage[misc]{ifsym}
\usepackage[colorlinks]{hyperref}

\hypersetup{citecolor=blue}
\usetikzlibrary{matrix,arrows,decorations.pathmorphing}

\newtheorem{theorem}{Theorem}[section]
\newtheorem{lemma}[theorem]{Lemma}
\newtheorem{corollary}[theorem]{Corollary}
\newtheorem{proposition}[theorem]{Proposition}

\newtheorem*{thmmain}{Main Theorem}

\theoremstyle{definition}
\newtheorem{definition}[theorem]{Definition}

\theoremstyle{remark}

\DeclareMathOperator{\cd}{cd}
\DeclareMathOperator{\Irr}{Irr}

\numberwithin{equation}{section}

\newcommand{\abs}[1]{\lvert#1\rvert}

\newcommand{\blankbox}[2]{%
  \parbox{\columnwidth}{\centering

    \setlength{\fboxsep}{0pt}%
    \fbox{\raisebox{0pt}[#2]{\hspace{#1}}}%
  }%
}

\makeatletter
\@namedef{subjclassname@2020}{%
  \textup{2020} Mathematics Subject Classification}
\makeatother

\begin{document}

\allowdisplaybreaks

\title{Classifying character degree graphs with seven vertices}

\author{Jacob Laubacher}
\address{Department of Mathematics, St. Norbert College, De Pere, WI 54115}
\email{jacob.laubacher@snc.edu}

\author{Mark Medwid}
\address{Department of Mathematical Sciences, Rhode Island College, Providence, RI 02908}
\email{mmedwid@ric.edu}

\author{Dylan Schuster}
\address{Department of Mathematics, St. Norbert College, De Pere, WI 54115}
\email{dylan.schuster@snc.edu}


\date{\today}

\keywords{Character degree graphs, Finite solvable groups, Classification.\\\indent\emph{Corresponding author.} Jacob Laubacher \Letter~\href{mailto:jacob.laubacher@snc.edu}{jacob.laubacher@snc.edu} \phone~920-403-2961.}

\thanks{The first and third authors received funding from the Poss-Wroble Fellowship at St. Norbert College}

\begin{abstract}
We study here the graphs with seven vertices in an effort to classify which of them appear as the prime character degree graphs of finite solvable groups. This classification is complete for the disconnected graphs. Of the 853 non-isomorphic connected graphs, we were able to demonstrate that twenty-two occur as prime character degree graphs. Two are of diameter three, while the remaining are constructed as direct products. Forty-four graphs remain unclassified. 
\end{abstract}

\maketitle

\section{Introduction}

In this paper we fix $G$ to be a finite solvable group. We write $\Irr(G)$ to denote the set of irreducible characters associated to $G$, and then by convention $\cd(G)=\{\chi(1)~|~\chi\in\Irr(G)\}$. One can then consider the set $\rho(G)$, which is defined to be all the prime numbers that divide some character in $\cd(G)$. The prime character degree graph of $G$, denoted $\Delta(G)$, is therefore the simple graph whose vertex set is $\rho(G)$. Thus, in this context, the words prime and vertex become synonymous. There is then an edge between two vertices $p$ and $q$ in $\Delta(G)$ if there exists some character $a\in\cd(G)$ such that $pq$ divides $a$. Prime character degree graphs have been studied quite extensively, both historically and recently (see \cite{I} for background or \cite{L4} for a comprehensive overview).

Prime character degree graphs can be considered for any group, but here we focus specifically on solvable groups, allowing the use of P\'alfy's condition. This landmark result from \cite{P} essentially mandates that any prime character degree graph that corresponds to a solvable group cannot have a triangle in its complement graph. This was generalized more recently in \cite{A} from a triangle to any odd cycle. We will make frequent use of these results to fulfill the goal of this paper, which is to study graphs that have seven vertices.

The main question we wish to answer is to determine which graphs with seven vertices can occur as the prime character degree graph of a finite solvable group. Classifications of this nature have been studied previously. Some of the first graphs to be considered were proven to be character degree graphs in \cite{Zh}, but one can see \cite{H2} where Huppert lists all of the graphs with four vertices or fewer that can arise as $\Delta(G)$ for some solvable group $G$. Lewis classifies all graphs with five vertices in \cite{L3} with one notable exception, and Bissler et al. then classify graphs with six vertices in \cite{BLL}. They had nine graphs that went unclassified in their paper.

Following \cite{MP}, one can list a staggering 853 connected graphs with seven vertices that are non-isomorphic. Using known results from the literature (like from \cite{A}), we can reduce this number to only eighty-five eligible graphs that can potentially occur as the prime character degree graph of some solvable group. After going through these remaining graphs, we determine that twenty-two of the connected graphs with seven vertices occur as $\Delta(G)$, as well as two disconnected graphs. After eliminating nineteen graphs, we are left with forty-four graphs that are still under consideration. Our results can be collated into the following:

\begin{thmmain}
The graphs with seven vertices that arise as $\Delta(G)$ for some solvable group $G$ are precisely those graphs in Figures \ref{figDO}, \ref{figA}, \ref{figBDT}, \ref{figB15}, \ref{figBDP}, and \ref{figCDP}, and possibly those in Figures \ref{figBM}, \ref{figC19}, and \ref{figCM}.
\end{thmmain}

\section{Preliminaries}

Many of the results that we employ throughout this article are well-known. In the case where we use certain work more frequently, we put it in this section, so as to stay as self-contained as possible while simultaneously making it easier on the reader. We shall first go through traditional results like P\'alfy's condition, and then review work that aids in classifications: results on disconnected graphs, results concerning graphs with diameter three, and results in relation to admissible vertices. We then round out the section with setting notation and conventions used throughout the paper.

\subsection{Families of graphs}

P\'alfy's condition, stated fully below, gives a great starting point for how we can show that some graphs do not occur as the prime character degree graph of any solvable group.

\begin{lemma}[P\'alfy's condition from \cite{P}]\label{PC}
Let $G$ be a solvable group and let $\pi$ be a set of primes contained in $\Delta(G)$. If $|\pi|=3$, then there exists an irreducible character of $G$ with degree divisible by at least two primes from $\pi$. (In other words, any three vertices of the prime character degree graph of a solvable group span at least one edge.)
\end{lemma}

In other words, P\'alfy's condition says that there must not be a $3$-cycle in the complement graph, denoted $\overline{\Delta}(G)$. Much more recently in \cite{A}, the authors prove a generalization of P\'aly's condition, extending from a $3$-cycle to any odd-cycle.

\begin{theorem}[Theorem A from \cite{A}]\label{genp}
Let $G$ be a finite solvable group. Then the graph $\overline{\Delta}(G)$ does not contain any cycle of odd length.
\end{theorem}

A powerful consequence is stated below, one that we will rely upon heavily in Section \ref{seccon}. First, however, it is important to understand the notion of a \emph{clique}, which is nothing more than a complete subgraph of a given graph.

\begin{corollary}[Corollary B from \cite{A}]\label{genpcor}
Let $G$ be a finite solvable group. Then $\rho(G)$ is covered by two subsets, each inducing a clique in $\Delta(G)$.
\end{corollary}

Next we recall the classification of the family of graphs $\{\Gamma_{k,t}\}$ from \cite{BL}. These graphs come up frequently in our classification of graphs with seven vertices, and therefore stating the construction and results seem worthwhile.

For $k\geq t\geq1$, the graph $\Gamma_{k,t}$ has two complete graphs $A$ and $B$ on $k$ and $t$ vertices, respectively. Letting $\rho(A) = \{a_1, a_2, \ldots, a_k\}$ and $\rho(B) = \{b_1, b_2, \ldots, b_t\}$ (with the restriction that $\rho(A)\cap\rho(B)=\varnothing$) we have that $a_i$ is adjacent to $b_i$ for all $1 \leq i \leq t$, while the remaining $k-t$ vertices in $A$ are adjacent to no vertices in $B$. These graphs were fully classified, the statement of the result of which is below.

\begin{theorem}\label{KT}\emph{(\cite{BL})}
The graph $\Gamma_{k,t}$ occurs as the prime character degree graph of a solvable group precisely when $t=1$ or $k=t=2$. Otherwise $\Gamma_{k,t}$ does not occur as the prime character degree graph of any solvable group.
\end{theorem}

\subsection{Disconnected graphs}

Disconnected graphs often arise. Beyond using the classification done in \cite{L} and one of the main results therein, we will also frequently use a result from P\'alfy. For completion's sake, we state both frequently used results here.

\begin{theorem}[Theorem 5.5 from \cite{L}]
Let $G$ be a solvable group and suppose that $\Delta(G)$ has two connected components. Then there is precisely one prime $p$ so that the Sylow $p$-subgroup of the Fitting subgroup of $G$ is not central in $G$.
\end{theorem}

\begin{theorem}[P\'alfy's inequality from \cite{P2}]\label{PI}
Let $G$ be a solvable group and $\Delta(G)$ its prime character degree graph. Suppose that $\Delta(G)$ is disconnected with two components having size $a$ and $b$, where $a\leq b$. Then $b\geq2^a-1$.
\end{theorem}

\subsection{Diameter three graphs}\label{secdiam3}

It is not uncommon that we will encounter a graph with seven vertices that has diameter three. These diameter three graphs were substantially tamed in \cite{S}, the groups investigated in \cite{C}, and an example ultimately constructed in \cite{L2}, and then done in more general in \cite{D}.

In particular to the specific result we will use, we recall that having a diameter of three in $\Delta(G)$ imposes a partition of $\rho(\Gamma)$ into four nonempty disjoint sets: $\rho(\Gamma)=\rho_1\cup\rho_2\cup\rho_3\cup\rho_4$. One can do this in the following way: find vertices $p$ and $q$ where the distance between them is three. Let $\rho_4$ be the set of all vertices that are distance three from the vertex $p$, which will include the vertex $q$. Let $\rho_3$ be the set of all vertices that are distance two from the vertex $p$. Let $\rho_2$ be the set of all vertices that are adjacent to the vertex $p$ and some vertex in $\rho_3$. Finally, let $\rho_1$ consist of $p$ and all vertices adjacent to $p$ that are not adjacent to anything in $\rho_3$. This notation is not unique, and depends on the vertices $p$ and $q$. Using the above notation, one can always arrange the four disjoint sets to have the following result.

\begin{proposition}\label{sassresult}\emph{(\cite{S})}
Let $G$ be a solvable group where $\Delta(G)$ has diameter three. One then has the following:
\begin{enumerate}[(i)]
    \item\label{sass1} $|\rho_3|\geq3$,
    \item\label{sass2} $|\rho_1\cup\rho_2|\leq|\rho_3\cup\rho_4|$,
    \item\label{sass3} if $|\rho_1\cup\rho_2|=n$, then $|\rho_3\cup\rho_4|\geq2^n$, and
    \item\label{sass4} $G$ has a normal Sylow $p$-subgroup for exactly one prime $p\in\rho_3$.
\end{enumerate}
\end{proposition}

\subsection{Admissible vertices}

In \cite{BL2}, Bissler and Lewis construct a family of graphs that do not occur as the prime character degree graph of any solvable group, where in their arguments they make use of labeling certain vertices as admissible. This has far-reaching consequences and is an incredible aid with the journey to proving that a certain graph does not occur as the prime character degree graph of any solvable group.
 
\begin{definition}(\cite{BL2})
Let $\Gamma$ be a graph and $p$ a vertex of $\Gamma$. Consider the following three conditions:
\begin{enumerate}[(i)]
\item the subgraph of $\Gamma$ obtained by removing $p$ and all edges incident to $p$ is non-occurring,
\item all of the subgraphs of $\Gamma$ obtained by removing one or more of the edges incident to $p$ are non-occurring,
\item all of the subgraphs of $\Gamma$ obtained by removing $p$, the edges incident to $p$, and one or more of the edges between two adjacent vertices of $p$ are non-occurring.
\end{enumerate}
If $p$ satisfies conditions (i) and (ii), then $p$ is said to be an \textbf{admissible} vertex. If $p$ satisfies all three conditions, then $p$ is said to be a \textbf{strongly admissible} vertex of $\Gamma$.
\end{definition}

Several immediate consequences follow, where in the latter two, it is to conclude that no normal nonabelian Sylow subgroup exists corresponding to a particular vertex. This argument is used heavily in Section \ref{sec3and4}.

\begin{lemma}\emph{(\cite{BL2})}
Let $G$ be a solvable group, and suppose $p$ is an admissible vertex of $\Delta(G)$. For every proper normal subgroup $N$ of $G$, suppose that $\Delta(N)$ is a proper subgraph of $\Delta(G)$. Then $O^p(G)=G$.
\end{lemma}

\begin{lemma}\label{strong}\emph{(\cite{BL2})}
Let $G$ be a solvable group and assume that $p$ is a prime whose vertex is a strongly admissible vertex of $\Delta(G)$. For every proper normal subgroup $N$ of $G$, suppose that $\Delta(G/N)$ is a proper subgraph of $\Delta(G)$. Then a Sylow $p$-subgroup of $G$ is not normal.
\end{lemma}

\begin{lemma}\label{pi}\emph{(\cite{BL2})}
Let $\Gamma$ be a graph satisfying P\'alfy's condition. Let $q$ be a vertex of $\Gamma$, and denote $\pi$ to be the set of vertices of $\Gamma$ adjacent to $q$, and $\rho$ to be the set of vertices of $\Gamma$ not adjacent to $q$. Assume that $\pi$ is the disjoint union of nonempty sets $\pi_1$ and $\pi_2$, and assume that no vertex in $\pi_1$ is adjacent in $\Gamma$ to any vertex in $\pi_2$. Let $v$ be a vertex in $\pi_2$ adjacent to an admissible vertex $s$ in $\rho$. Furthermore, assume there exists another vertex $w$ in $\rho$ that is not adjacent to $v$.

Let $G$ be a solvable group such that $\Delta(G)=\Gamma$, and assume that for every proper normal subgroup $N$ of $G$, $\Delta(N)$ is a proper subgraph of $\Delta(G)$. Then a Sylow $q$-subgroup of $G$ for the prime associated to $q$ is not normal.
\end{lemma}

Lastly, we state the result which will often lead us to a final contradiction. It once again makes use of admissible vertices.

\begin{lemma}\label{thefinal}\emph{(\cite{BL2})}
Let $\Gamma$ be a graph satisfying P\'alfy's condition with $n\geq5$ vertices. Also, assume there exist distinct vertices $a$ and $b$ of $\Gamma$ such that $a$ is adjacent to an admissible vertex $c$, $b$ is not adjacent to $c$, and $a$ is not adjacent to an admissible vertex $d$.

Let $G$ be a solvable group and suppose for all proper normal subgroups $N$ of $G$ we have that $\Delta(N)$ and $\Delta(G/N)$ are proper subgraphs of $\Gamma$. Let $F$ be the Fitting subgroup of $G$ and suppose that $F$ is minimal normal in $G$. Then $\Gamma$ is not the prime character degree graph of any solvable group.
\end{lemma}

\subsection{Notation and conventions}

To simplify language throughout this paper, we will adopt the shorthand of referring to a graph $\Gamma$ as \emph{occurring} if $\Gamma=\Delta(G)$ for some finite solvable group $G$, and \emph{non-occurring} otherwise.

Moreover, following \cite{DLM}, we shall make use of the notation $\Gamma[p]$ to denote the subgraph of $\Gamma$ obtained by deleting $p$ and its incident edges from $\Gamma$. Furthermore, we will use the expression $\epsilon(p,q)$ to refer to the edge joining $p$ and $q$, and therefore $\Gamma[\epsilon(p,q)]$ denotes the subgraph of $\Gamma$ obtained by deleting the edge between $p$ and $q$, but not the vertices themselves. We can then combine these notations, where for example $\Gamma[p_1,\epsilon(p_2,p_3)]$ would be the subgraph of $\Gamma$ obtained by deleting $p_1$ and its incident edges, along with the edge connecting $p_2$ and $p_3$. Making use of this notation is helpful when referring to subgraphs, which is done quite frequently when showing that a certain vertex is admissible. Specifically, we will make use of this often in Section \ref{sec3and4}.

Finally, when referring to a complete graph on $n$ vertices, we will use the common notation of $K_n$. It is well-known that $K_n$ occurs as the prime character degree graph for some solvable group for all $n\in\mathbb{N}$. This will be useful in Section \ref{sec1and6} where we often use $K_n$ to construct a new group via a direct product. For completion's sake, one can recall that the direct product between two groups, denoted $H\times K$, results in their vertex set as $\rho(H\times K)=\rho(H)\cup\rho(K)$. Since there are infinitely many primes for which these groups have their respective graphs, it is no problem requiring that $\rho(H)$ and $\rho(K)$ be disjoint. There is then an edge between $p$ and $q$ in $\Delta(H\times K)$ if either there was an edge between $p$ and $q$ in $\Delta(H)$ or $\Delta(K)$, or if $p\in\rho(H)$ and $q\in\rho(K)$. Direct products are a great way to build new groups and graphs from old ones. In fact, the singleton $K_1$ can be easily constructed as $\Delta(G)$ for some finite solvable group $G$, and then an induction argument yields $K_n=K_1\times K_{n-1}$.

\section{Disconnected Graphs}\label{secD}

In this section, we fully classify the disconnected graphs with seven vertices. Due to P\'alfy's condition, if a prime character degree graph is disconnected, then there can only be two components. Furthermore, each component must be complete. This leaves us with three cases concerning the sizes of the two complete components: (1) one and six, (2) two and five, and (3) three and four.

We will show, by construction, that only the disconnected graphs in Figure \ref{figDO} occur as the prime character degree graph of a solvable group.

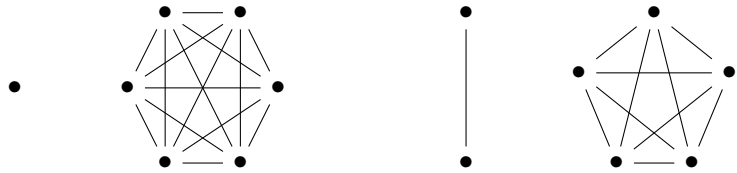
\begin{figure}[htb]
    \centering
$
\begin{tikzpicture}[scale=2]
\node (1a) at (0,.5) {$\bullet$};
\node (2a) at (.75,.5) {$\bullet$};
\node (3a) at (1,1) {$\bullet$};
\node (4a) at (1,0) {$\bullet$};
\node (5a) at (1.5,1) {$\bullet$};
\node (6a) at (1.5,0) {$\bullet$};
\node (7a) at (1.75,.5) {$\bullet$};
\path[font=\small,>=angle 90]
(2a) edge node [right] {$ $} (3a)
(2a) edge node [right] {$ $} (4a)
(2a) edge node [right] {$ $} (5a)
(2a) edge node [right] {$ $} (6a)
(2a) edge node [right] {$ $} (7a)
(3a) edge node [right] {$ $} (4a)
(3a) edge node [right] {$ $} (5a)
(3a) edge node [right] {$ $} (6a)
(3a) edge node [right] {$ $} (7a)
(4a) edge node [right] {$ $} (5a)
(4a) edge node [right] {$ $} (6a)
(4a) edge node [right] {$ $} (7a)
(5a) edge node [right] {$ $} (6a)
(5a) edge node [right] {$ $} (7a)
(6a) edge node [right] {$ $} (7a);
\node (1b) at (3,1) {$\bullet$};
\node (2b) at (3,0) {$\bullet$};
\node (3b) at (3.75,.6) {$\bullet$};
\node (4b) at (4,0) {$\bullet$};
\node (5b) at (4.5,0) {$\bullet$};
\node (6b) at (4.25,1) {$\bullet$};
\node (7b) at (4.75,.6) {$\bullet$};
\path[font=\small,>=angle 90]
(1b) edge node [right] {$ $} (2b)
(3b) edge node [right] {$ $} (4b)
(3b) edge node [right] {$ $} (5b)
(3b) edge node [right] {$ $} (6b)
(3b) edge node [right] {$ $} (7b)
(4b) edge node [right] {$ $} (5b)
(4b) edge node [right] {$ $} (6b)
(4b) edge node [right] {$ $} (7b)
(5b) edge node [right] {$ $} (6b)
(5b) edge node [right] {$ $} (7b)
(6b) edge node [right] {$ $} (7b);
\end{tikzpicture}
$
    \caption{Disconnected graphs that occur}
    \label{figDO}
\end{figure}

\subsection{Constructing case (1)}\label{secD1}

Following the constructions seen in \cite{BLL} or \cite{L3}, we begin by considering the field of order $2^{81}$ acted on by its full multiplication group and then its Galois group. Set the resulting solvable group as $G_1$. Noting that $2^{81}-1$ is the product of six distinct prime numbers, namely $2^{81}-1=7\cdot73\cdot2593\cdot71119\cdot262657\cdot97685839$, we get that
$$
\cd(G_1)=\{1,3,9,27,81,2^{81}-1\}.
$$
It is easy to see that $\Delta(G_1)$ is the disconnected graph with two complete components, as seen in Figure \ref{figDO1}. For ease of notation, set $p=71119$, $q=262657$, and $r=97685839$.

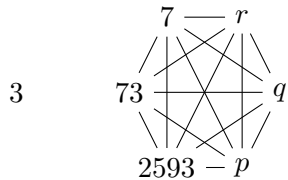
\begin{figure}[htb]
    \centering
$
\begin{tikzpicture}[scale=2]
\node (1a) at (0,.5) {$3$};
\node (2a) at (.75,.5) {$73$};
\node (3a) at (1,1) {$7$};
\node (4a) at (1,0) {$2593$};
\node (5a) at (1.5,1) {$r$};
\node (6a) at (1.5,0) {$p$};
\node (7a) at (1.75,.5) {$q$};
\path[font=\small,>=angle 90]
(2a) edge node [right] {$ $} (3a)
(2a) edge node [right] {$ $} (4a)
(2a) edge node [right] {$ $} (5a)
(2a) edge node [right] {$ $} (6a)
(2a) edge node [right] {$ $} (7a)
(3a) edge node [right] {$ $} (4a)
(3a) edge node [right] {$ $} (5a)
(3a) edge node [right] {$ $} (6a)
(3a) edge node [right] {$ $} (7a)
(4a) edge node [right] {$ $} (5a)
(4a) edge node [right] {$ $} (6a)
(4a) edge node [right] {$ $} (7a)
(5a) edge node [right] {$ $} (6a)
(5a) edge node [right] {$ $} (7a)
(6a) edge node [right] {$ $} (7a);
\end{tikzpicture}
$
    \caption{The graph $\Delta(G_1)$}
    \label{figDO1}
\end{figure}

\subsection{Constructing case (2)}\label{secD2}

Similarly, now consider the field of order $2^{51}$ acted on by its full multiplication group and then its Galois group. Let this solvable group be $G_2$. Observe that $2^{51}-1$ is the product of five distinct prime numbers: in particular, $2^{51}-1=7\cdot103\cdot2143\cdot11119\cdot131071$. One can compute that
$$
\cd(G_2)=\{1,3,17,51,2^{51}-1\}.
$$
Setting $p=11119$ and $q=131071$, one can see the resulting prime character degree graph in Figure \ref{figDO2}.

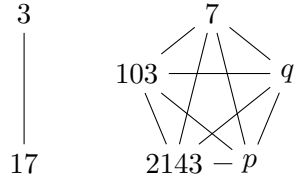
\begin{figure}[htb]
    \centering
$
\begin{tikzpicture}[scale=2]
\node (1b) at (3,1) {$3$};
\node (2b) at (3,0) {$17$};
\node (3b) at (3.75,.6) {$103$};
\node (4b) at (4,0) {$2143$};
\node (5b) at (4.5,0) {$p$};
\node (6b) at (4.25,1) {$7$};
\node (7b) at (4.75,.6) {$q$};
\path[font=\small,>=angle 90]
(1b) edge node [right] {$ $} (2b)
(3b) edge node [right] {$ $} (4b)
(3b) edge node [right] {$ $} (5b)
(3b) edge node [right] {$ $} (6b)
(3b) edge node [right] {$ $} (7b)
(4b) edge node [right] {$ $} (5b)
(4b) edge node [right] {$ $} (6b)
(4b) edge node [right] {$ $} (7b)
(5b) edge node [right] {$ $} (6b)
(5b) edge node [right] {$ $} (7b)
(6b) edge node [right] {$ $} (7b);
\end{tikzpicture}
$
    \caption{The graph $\Delta(G_2)$}
    \label{figDO2}
\end{figure}

\subsection{Eliminating case (3)}\label{secD3}

It is clear that the disconnected graph in Figure \ref{figDDNO} does not occur as the prime character degree graph of any solvable group. This is due to P\'alfy's inequality (see Theorem \ref{PI}) by taking $a=3$ and $b=4$. Notice then that
$$
4=b\geq2^a-1=2^3-1=7,
$$
a contradiction.

\begin{figure}[htb]
    \centering
$
\begin{tikzpicture}[scale=2]
\node (1) at (0,0) {$\bullet$};
\node (2) at (0,1) {$\bullet$};
\node (3) at (.5,.5) {$\bullet$};
\node (4) at (1.25,.5) {$\bullet$};
\node (5) at (1.75,1) {$\bullet$};
\node (6) at (1.75,0) {$\bullet$};
\node (7) at (2.25,.5) {$\bullet$};
\path[font=\small,>=angle 90]
(1) edge node [right] {$ $} (2)
(1) edge node [right] {$ $} (3)
(2) edge node [right] {$ $} (3)
(4) edge node [right] {$ $} (5)
(4) edge node [right] {$ $} (6)
(4) edge node [right] {$ $} (7)
(5) edge node [right] {$ $} (6)
(5) edge node [right] {$ $} (7)
(6) edge node [right] {$ $} (7);
\end{tikzpicture}
$
    \caption{Disconnected graph that does not occur}
    \label{figDDNO}
\end{figure}
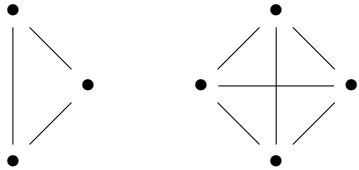

\section{Connected Graphs}\label{seccon}

There are $853$ non-isomorphic simple connected graphs with seven vertices, as per \cite{MP}. Due to the powerful landmark result from \cite{A} (see Corollary \ref{genpcor}), we know that if a graph $\Gamma$ is going to occur as the prime character degree graph of some solvable group, then the vertex set $\rho(\Gamma)$ must be covered by two subsets, each forming a complete subgraph (a clique). Following the above, we are able to immediately reduce the possibilities to eighty-five eligible graphs. Much like the disconnected cases from Section \ref{secD}, we too have three options here for the connected case. A connected graph with seven vertices, if it is to occur, must therefore be covered by cliques of one and six, two and five, or three and four.

We then organize these remaining eighty-five graphs into one of the three Appendices \ref{AA}, \ref{AB}, or \ref{AC}. Each graph is placed into the appendix of the largest possible clique. For example, the graph $A_1$ is covered by cliques of one and six, but also covered by cliques of two and five. But since we want to place it into exactly one category, we choose to place it into the highest available clique: one and six, instead of two and five.

These remaining eighty-five graphs, the complete lists of which are organized by edge count within each Appendix, will then be discussed in Sections \ref{sec1and6}, \ref{sec2and5}, and \ref{sec3and4}, respectively. In particular, Appendix \ref{AA} has six graphs, Appendix \ref{AB} houses twenty-six graphs, while Appendix \ref{AC} contains fifty-three graphs, totalling the aforementioned eighty-five. As we shall see in the sections that follow, of these eighty-five, we conclude that twenty-two of these graphs do indeed occur as $\Delta(G)$ for some finite solvable group $G$, whereas nineteen graphs do not occur as the prime character degree graph of any solvable group. This leaves us with forty-four graphs that are yet-to-be classified, and therefore possibly occur.

\section{Cliques of One and Six}\label{sec1and6}

There are six connected graphs with cliques of one and six, all of which we fully classify. One can see Appendix \ref{AA} for the complete list. In fact, each graph is a direct product of graphs that have been shown to occur in previous papers, or graphs that have already been studied.

It is not difficult to see that $A_1$ is the graph $\Gamma_{6,1}$ from \cite{BL} (see Theorem \ref{KT}). Next, the graph $A_i$ can be realized as the direct product between $\Gamma_{7-i,1}$ and $K_{i-1}$ for all $2\leq i\leq 5$. Finally, the graph $A_6$ is exactly $K_7$. Hence, all the graphs in Figure \ref{figA} occur as the prime character degree graph of a solvable group.

\begin{figure}[htb]
    \centering
$
\begin{tikzpicture}[scale=2]
\node (1a) at (0,.5) {$\bullet$};
\node (2a) at (.75,.5) {$\bullet$};
\node (3a) at (1,1) {$\bullet$};
\node (4a) at (1,0) {$\bullet$};
\node (5a) at (1.5,1) {$\bullet$};
\node (6a) at (1.5,0) {$\bullet$};
\node (7a) at (1.75,.5) {$\bullet$};
\path[font=\small,>=angle 90]
(2a) edge node [right] {$ $} (3a)
(2a) edge node [right] {$ $} (4a)
(2a) edge node [right] {$ $} (5a)
(2a) edge node [right] {$ $} (6a)
(2a) edge node [right] {$ $} (7a)
(3a) edge node [right] {$ $} (4a)
(3a) edge node [right] {$ $} (5a)
(3a) edge node [right] {$ $} (6a)
(3a) edge node [right] {$ $} (7a)
(4a) edge node [right] {$ $} (5a)
(4a) edge node [right] {$ $} (6a)
(4a) edge node [right] {$ $} (7a)
(5a) edge node [right] {$ $} (6a)
(5a) edge node [right] {$ $} (7a)
(6a) edge node [right] {$ $} (7a)
(1a) edge node [right] {$ $} (2a);
\node (1b) at (2.5,.5) {$\bullet$};
\node (2b) at (3.25,.7) {$\bullet$};
\node (3b) at (3.25,.3) {$\bullet$};
\node (4b) at (3.75,1) {$\bullet$};
\node (5b) at (3.75,0) {$\bullet$};
\node (6b) at (4.25,.7) {$\bullet$};
\node (7b) at (4.25,.3) {$\bullet$};
\path[font=\small,>=angle 90]
(2b) edge node [right] {$ $} (3b)
(2b) edge node [right] {$ $} (4b)
(2b) edge node [right] {$ $} (5b)
(2b) edge node [right] {$ $} (6b)
(2b) edge node [right] {$ $} (7b)
(3b) edge node [right] {$ $} (4b)
(3b) edge node [right] {$ $} (5b)
(3b) edge node [right] {$ $} (6b)
(3b) edge node [right] {$ $} (7b)
(4b) edge node [right] {$ $} (5b)
(4b) edge node [right] {$ $} (6b)
(4b) edge node [right] {$ $} (7b)
(5b) edge node [right] {$ $} (6b)
(5b) edge node [right] {$ $} (7b)
(6b) edge node [right] {$ $} (7b)
(1b) edge node [right] {$ $} (2b)
(1b) edge node [right] {$ $} (3b);
\node (1c) at (5,.5) {$\bullet$};
\node (2c) at (5.75,.5) {$\bullet$};
\node (3c) at (6,1) {$\bullet$};
\node (4c) at (6,0) {$\bullet$};
\node (5c) at (6.5,1) {$\bullet$};
\node (6c) at (6.5,0) {$\bullet$};
\node (7c) at (6.75,.5) {$\bullet$};
\path[font=\small,>=angle 90]
(2c) edge node [right] {$ $} (3c)
(2c) edge node [right] {$ $} (4c)
(2c) edge node [right] {$ $} (5c)
(2c) edge node [right] {$ $} (6c)
(2c) edge node [right] {$ $} (7c)
(3c) edge node [right] {$ $} (4c)
(3c) edge node [right] {$ $} (5c)
(3c) edge node [right] {$ $} (6c)
(3c) edge node [right] {$ $} (7c)
(4c) edge node [right] {$ $} (5c)
(4c) edge node [right] {$ $} (6c)
(4c) edge node [right] {$ $} (7c)
(5c) edge node [right] {$ $} (6c)
(5c) edge node [right] {$ $} (7c)
(6c) edge node [right] {$ $} (7c)
(1c) edge node [right] {$ $} (2c)
(1c) edge node [right] {$ $} (3c)
(1c) edge node [right] {$ $} (4c);
\node (1d) at (0,-1) {$\bullet$};
\node (2d) at (.75,-.8) {$\bullet$};
\node (3d) at (.75,-1.2) {$\bullet$};
\node (4d) at (1.25,-.5) {$\bullet$};
\node (5d) at (1.25,-1.5) {$\bullet$};
\node (6d) at (1.75,-.8) {$\bullet$};
\node (7d) at (1.75,-1.2) {$\bullet$};
\path[font=\small,>=angle 90]
(2d) edge node [right] {$ $} (3d)
(2d) edge node [right] {$ $} (4d)
(2d) edge node [right] {$ $} (5d)
(2d) edge node [right] {$ $} (6d)
(2d) edge node [right] {$ $} (7d)
(3d) edge node [right] {$ $} (4d)
(3d) edge node [right] {$ $} (5d)
(3d) edge node [right] {$ $} (6d)
(3d) edge node [right] {$ $} (7d)
(4d) edge node [right] {$ $} (5d)
(4d) edge node [right] {$ $} (6d)
(4d) edge node [right] {$ $} (7d)
(5d) edge node [right] {$ $} (6d)
(5d) edge node [right] {$ $} (7d)
(6d) edge node [right] {$ $} (7d)
(1d) edge node [right] {$ $} (2d)
(1d) edge node [right] {$ $} (3d)
(1d) edge node [right] {$ $} (4d)
(1d) edge node [right] {$ $} (5d);
\node (1e) at (2.5,-1) {$\bullet$};
\node (2e) at (3.25,-1) {$\bullet$};
\node (3e) at (3.5,-.5) {$\bullet$};
\node (4e) at (3.5,-1.5) {$\bullet$};
\node (5e) at (4,-.5) {$\bullet$};
\node (6e) at (4,-1.5) {$\bullet$};
\node (7e) at (4.25,-1) {$\bullet$};
\path[font=\small,>=angle 90]
(2e) edge node [right] {$ $} (3e)
(2e) edge node [right] {$ $} (4e)
(2e) edge node [right] {$ $} (5e)
(2e) edge node [right] {$ $} (6e)
(2e) edge node [right] {$ $} (7e)
(3e) edge node [right] {$ $} (4e)
(3e) edge node [right] {$ $} (5e)
(3e) edge node [right] {$ $} (6e)
(3e) edge node [right] {$ $} (7e)
(4e) edge node [right] {$ $} (5e)
(4e) edge node [right] {$ $} (6e)
(4e) edge node [right] {$ $} (7e)
(5e) edge node [right] {$ $} (6e)
(5e) edge node [right] {$ $} (7e)
(6e) edge node [right] {$ $} (7e)
(1e) edge node [right] {$ $} (2e)
(1e) edge node [right] {$ $} (3e)
(1e) edge node [right] {$ $} (4e)
(1e) edge node [right] {$ $} (5e)
(1e) edge node [right] {$ $} (6e);
\node (1f) at (5,-1) {$\bullet$};
\node (2f) at (5.75,-.8) {$\bullet$};
\node (3f) at (5.75,-1.2) {$\bullet$};
\node (4f) at (6.25,-.5) {$\bullet$};
\node (5f) at (6.25,-1.5) {$\bullet$};
\node (6f) at (6.75,-.8) {$\bullet$};
\node (7f) at (6.75,-1.2) {$\bullet$};
\path[font=\small,>=angle 90]
(2f) edge node [right] {$ $} (3f)
(2f) edge node [right] {$ $} (4f)
(2f) edge node [right] {$ $} (5f)
(2f) edge node [right] {$ $} (6f)
(2f) edge node [right] {$ $} (7f)
(3f) edge node [right] {$ $} (4f)
(3f) edge node [right] {$ $} (5f)
(3f) edge node [right] {$ $} (6f)
(3f) edge node [right] {$ $} (7f)
(4f) edge node [right] {$ $} (5f)
(4f) edge node [right] {$ $} (6f)
(4f) edge node [right] {$ $} (7f)
(5f) edge node [right] {$ $} (6f)
(5f) edge node [right] {$ $} (7f)
(6f) edge node [right] {$ $} (7f)
(1f) edge node [right] {$ $} (2f)
(1f) edge node [right] {$ $} (3f)
(1f) edge node [right] {$ $} (4f)
(1f) edge node [right] {$ $} (5f)
(1f) edge node [right] {$ $} (6f)
(1f) edge node [right] {$ $} (7f);
\end{tikzpicture}
$
    \caption{The graphs $A_i$ for $1\leq i\leq6$}
    \label{figA}
\end{figure}
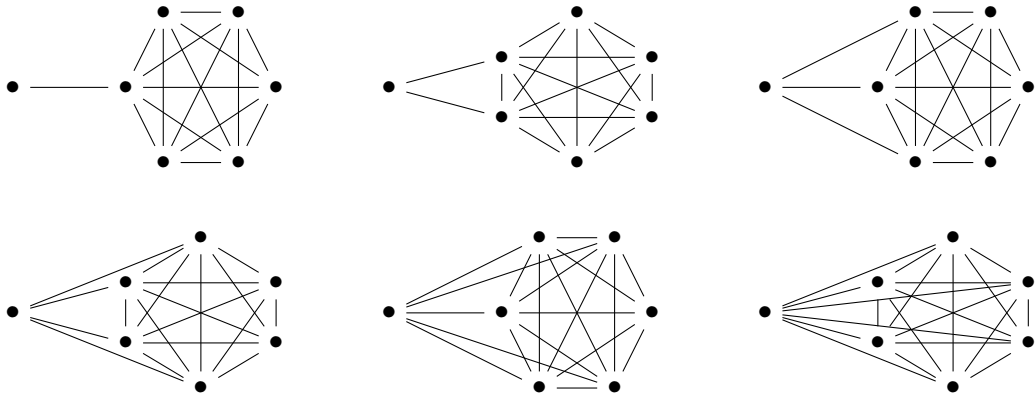

\section{Cliques of Two and Five}\label{sec2and5}

There are twenty-six connected graphs whose vertex sets are covered by cliques of two and five, where one can see Appendix \ref{AB} for the complete list. In this section we will classify fifteen of these graphs, leaving eleven unknown.

Thus, the goal of this section is to show that all the graphs in Figures \ref{figBDT}, \ref{figB15}, and \ref{figBDP} occur, and that all the graphs in Figure \ref{figBM} possibly occur. Otherwise, as we shall see, the graphs in Figure \ref{figBDNO} do not occur as the prime character degree graph of any solvable group.

\begin{figure}[htb]
    \centering
$
\begin{tikzpicture}[scale=2]
\node (1a) at (0,1) {$\bullet$};
\node (2a) at (0,0) {$\bullet$};
\node (3a) at (.75,.8) {$\bullet$};
\node (4a) at (.75,.2) {$\bullet$};
\node (5a) at (1.35,1) {$\bullet$};
\node (6a) at (1.35,0) {$\bullet$};
\node (7a) at (1.75,.5) {$\bullet$};
\path[font=\small,>=angle 90]
(1a) edge node [right] {$ $} (2a)
(3a) edge node [right] {$ $} (4a)
(3a) edge node [right] {$ $} (5a)
(3a) edge node [right] {$ $} (6a)
(3a) edge node [right] {$ $} (7a)
(4a) edge node [right] {$ $} (5a)
(4a) edge node [right] {$ $} (6a)
(4a) edge node [right] {$ $} (7a)
(5a) edge node [right] {$ $} (6a)
(5a) edge node [right] {$ $} (7a)
(6a) edge node [right] {$ $} (7a)
(1a) edge node [right] {$ $} (3a)
(1a) edge node [right] {$ $} (4a)
(1a) edge node [right] {$ $} (5a);
\node (1b) at (2.75,1) {$\bullet$};
\node (2b) at (2.75,0) {$\bullet$};
\node (3b) at (3.5,.8) {$\bullet$};
\node (4b) at (3.5,.2) {$\bullet$};
\node (5b) at (4.1,1) {$\bullet$};
\node (6b) at (4.1,0) {$\bullet$};
\node (7b) at (4.5,.5) {$\bullet$};
\path[font=\small,>=angle 90]
(1b) edge node [right] {$ $} (2b)
(3b) edge node [right] {$ $} (4b)
(3b) edge node [right] {$ $} (5b)
(3b) edge node [right] {$ $} (6b)
(3b) edge node [right] {$ $} (7b)
(4b) edge node [right] {$ $} (5b)
(4b) edge node [right] {$ $} (6b)
(4b) edge node [right] {$ $} (7b)
(5b) edge node [right] {$ $} (6b)
(5b) edge node [right] {$ $} (7b)
(6b) edge node [right] {$ $} (7b)
(1b) edge node [right] {$ $} (3b)
(1b) edge node [right] {$ $} (4b)
(1b) edge node [right] {$ $} (5b)
(1b) edge node [right] {$ $} (6b);
\end{tikzpicture}
$
    \caption{The graphs $B_3$ and $B_4$}
    \label{figBDT}
\end{figure}
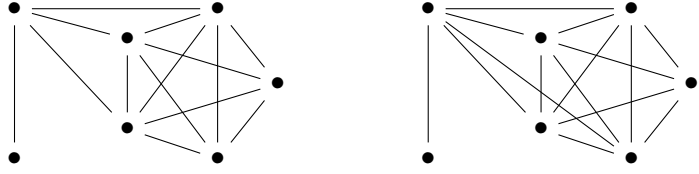

The graphs in Figure \ref{figBDT} will be constructed below, but earn special attention as they both have diameter three, and are not merely constructed as direct products.

\subsection{Constructing the graphs $B_3$ and $B_4$}\label{secBB3}

The graphs $B_3$ and $B_4$ result from a direct application of the techniques found in \cite{D}, which expand on the construction found in \cite{L2}. These works explicitly construct solvable groups for which the character degree graph has diameter three. Of particular note in \cite{D} is that the resulting character degree graph is easy to calculate directly. We begin by choosing primes $p$, $q$, and $r$ that satisfy certain divisibility conditions. 

Specifically, we allow $p$ to be any prime. From here, we require that $q$ and $r$ are chosen so that $q < r$ and the quantities $\Phi_{q}(p)$, $\Phi_{r}(p)$, and $\Phi_{qr}(p)$ are pairwise coprime, with $\Phi_{n}(x)$ being the $n$th cyclotomic polynomial. As illustrated in \cite{D}, the computation of character degrees is substantially more complicated when one chooses $q$ to be a prime larger than 3. Fortunately, the case $q = 3$ is sufficient to construct both $B_3$ and $B_4$.

To construct $G$, we begin with a field $F$ of order $p^{3r}$. We then consider $F\{ X \}$, the skew polynomial ring, in which
$$
X \alpha = \alpha^p X
$$
for all $\alpha \in F$. Next, consider the quotient ring, $R$, of $F\{X\}$ and the ideal generated by $X^{4}$. Let $J$ be the Jacobson radical of $R$, and note that elements of $J$ all have the form $\alpha_1 x + \alpha_2 x^2 + \alpha_3 x^3$, where $x$ is the image of the indeterminate $X$ in $R$. Then, define $P = 1 + J$, which is a group of order $p^{9r}$. We next let $C$ be the subgroup of $F^\times$ that has order $\frac{p^{3r}-1}{p-1}$, and define 
$$
T = P \rtimes_{\varphi} C,
$$
where $c \in C$ acts on $P$ via the automorphism defined by:
$$
\varphi_c (1 + \alpha_1 x + \alpha_2 x^2 + \alpha_3 x^3) = 1 + \alpha_1 c x + \alpha_2 c^{p+1} x^2 + \alpha_3 c^{p^2 + p + 1} x^3.
$$
Finally, $G$ is the semidirect product of $T$ and $\mathcal{G}$, with $\mathcal{G}$ being the Galois group of $F$ over its prime subfield. One can see that $|G| = 3rp^{9r} \left( \frac{p^{3r}-1}{p-1} \right)$, so the constructed group is indeed finite. 

To construct $B_3$ and $B_4$, we impose the additional condition that the quantity $\frac{p^{3r} - 1}{p-1} = stuv$, with $s$, $t$, $u$, and $v$ being prime and distinct from $p$, $q$, and $r$; this helps to ensure a seven-vertex character degree graph. The argument in \cite{D} yields the following set of character degrees:
\begin{align*}
\cd(G) = \left\{ 1, 3, r, 3r, stuv, p^{\frac{3r-1}{2}} stuv, 3p^{3r} \left( \frac{stuv}{p^2 + p + 1} \right), p^{3r-3} \left( \frac{stuv}{p^2 + p + 1} \right),\right. \\ \left. 3p^{3r-3} \left( \frac{stuv}{p^2 + p + 1} \right), p^{3r-3} stuv, p^{3r-2} stuv \right\}.
\end{align*}

Let $G_3$ correspond to the choices $p = 23$, $q=3$ and $r = 5$, which gives $s = 7$, $t = 79$, $u = 292561$, and $v = 74912328481$. In this case, we get
$$
\cd(G_3) = \{1, 3, 5, 15, stuv, 23^7 stuv, 3 \cdot 23^{15} uv, 23^{12} uv, 3 \cdot 23^{12} uv, 23^{12} stuv, 23^{13} stuv\}.
$$
Similarly, let $G_4$ correspond to the choices $p = 2$, $q = 3$, and $r = 11$, giving $s = 7$, $t = 23$, $u = 89$, and $v = 599479$. This yields
$$
\cd(G_4) = \{1, 3, 11, 33, stuv, 2^{16} stuv, 3 \cdot 2^{33} tuv, 2^{30} tuv, 3 \cdot 2^{30} tuv, 2^{30} stuv, 2^{31} stuv \}.
$$
We can see $\Delta(G_3)$ and $\Delta(G_4)$ in Figure \ref{figBCon}, which are clearly $B_3$ and $B_4$, respectively.

\begin{figure}[htb]
    \centering
$
\begin{tikzpicture}[scale=2]
\node (1a) at (0,1) {$3$};
\node (2a) at (0,0) {$5$};
\node (3a) at (.75,.8) {$v$};
\node (4a) at (.75,.2) {$u$};
\node (5a) at (1.35,1) {$23$};
\node (6a) at (1.35,0) {$t$};
\node (7a) at (1.75,.5) {$s$};
\path[font=\small,>=angle 90]
(1a) edge node [right] {$ $} (2a)
(3a) edge node [right] {$ $} (4a)
(3a) edge node [right] {$ $} (5a)
(3a) edge node [right] {$ $} (6a)
(3a) edge node [right] {$ $} (7a)
(4a) edge node [right] {$ $} (5a)
(4a) edge node [right] {$ $} (6a)
(4a) edge node [right] {$ $} (7a)
(5a) edge node [right] {$ $} (6a)
(5a) edge node [right] {$ $} (7a)
(6a) edge node [right] {$ $} (7a)
(1a) edge node [right] {$ $} (3a)
(1a) edge node [right] {$ $} (4a)
(1a) edge node [right] {$ $} (5a);
\node (1b) at (2.75,1) {$3$};
\node (2b) at (2.75,0) {$11$};
\node (3b) at (3.5,.8) {$v$};
\node (4b) at (3.5,.2) {$u$};
\node (5b) at (4.1,1) {$2$};
\node (6b) at (4.1,0) {$t$};
\node (7b) at (4.5,.5) {$s$};
\path[font=\small,>=angle 90]
(1b) edge node [right] {$ $} (2b)
(3b) edge node [right] {$ $} (4b)
(3b) edge node [right] {$ $} (5b)
(3b) edge node [right] {$ $} (6b)
(3b) edge node [right] {$ $} (7b)
(4b) edge node [right] {$ $} (5b)
(4b) edge node [right] {$ $} (6b)
(4b) edge node [right] {$ $} (7b)
(5b) edge node [right] {$ $} (6b)
(5b) edge node [right] {$ $} (7b)
(6b) edge node [right] {$ $} (7b)
(1b) edge node [right] {$ $} (3b)
(1b) edge node [right] {$ $} (4b)
(1b) edge node [right] {$ $} (5b)
(1b) edge node [right] {$ $} (6b);
\end{tikzpicture}
$
    \caption{The graphs $\Delta(G_3)=B_3$ and $\Delta(G_4)=B_4$}
    \label{figBCon}
\end{figure}
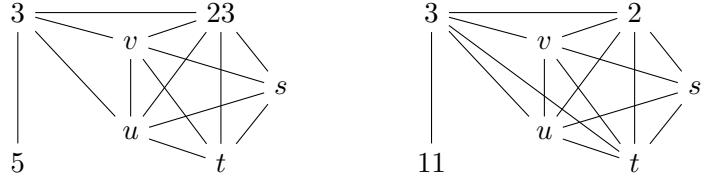

\subsection{Direct products}\label{secBDP}

There are ten graphs in Appendix \ref{AB} that are direct products of other graphs that are known to occur. These graphs are $B_{15}$, shown in Figure \ref{figB15}, and $B_i$ for $i\in\mathcal{I}$ where
$$
\mathcal{I}:=\{6,13,14,16,19,21,23,24,26\},
$$
seen in Figure \ref{figBDP}. 

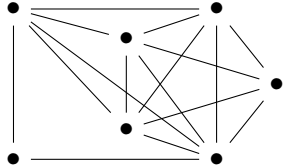
\begin{figure}[htb]
    \centering
$
\begin{tikzpicture}[scale=2]
\node (1) at (0,1) {$\bullet$};
\node (2) at (0,0) {$\bullet$};
\node (3) at (.75,.8) {$\bullet$};
\node (4) at (.75,.2) {$\bullet$};
\node (5) at (1.35,1) {$\bullet$};
\node (6) at (1.35,0) {$\bullet$};
\node (7) at (1.75,.5) {$\bullet$};
\path[font=\small,>=angle 90]
(1) edge node [right] {$ $} (2)
(3) edge node [right] {$ $} (4)
(3) edge node [right] {$ $} (5)
(3) edge node [right] {$ $} (6)
(3) edge node [right] {$ $} (7)
(4) edge node [right] {$ $} (5)
(4) edge node [right] {$ $} (6)
(4) edge node [right] {$ $} (7)
(5) edge node [right] {$ $} (6)
(5) edge node [right] {$ $} (7)
(6) edge node [right] {$ $} (7)
(1) edge node [right] {$ $} (3)
(1) edge node [right] {$ $} (4)
(1) edge node [right] {$ $} (5)
(1) edge node [right] {$ $} (6)
(2) edge node [right] {$ $} (6);
\end{tikzpicture}
$
    \caption{The graph $B_{15}$}
    \label{figB15}
\end{figure}

The graph $B_{15}$ deserves some special attention. It can be constructed as the direct product between the singleton $K_1$ and the diameter three graph with six vertices, which was constructed in \cite{L2} and classified in \cite{BLL}.

\begin{figure}[htb]
    \centering
$
\begin{tikzpicture}[scale=2]
\node (1a) at (0,1) {$\bullet$};
\node (2a) at (0,0) {$\bullet$};
\node (3a) at (.75,.5) {$\bullet$};
\node (4a) at (1.15,1) {$\bullet$};
\node (5a) at (1.15,0) {$\bullet$};
\node (6a) at (1.75,.8) {$\bullet$};
\node (7a) at (1.75,.2) {$\bullet$};
\path[font=\small,>=angle 90]
(1a) edge node [right] {$ $} (2a)
(3a) edge node [right] {$ $} (4a)
(3a) edge node [right] {$ $} (5a)
(3a) edge node [right] {$ $} (6a)
(3a) edge node [right] {$ $} (7a)
(4a) edge node [right] {$ $} (5a)
(4a) edge node [right] {$ $} (6a)
(4a) edge node [right] {$ $} (7a)
(5a) edge node [right] {$ $} (6a)
(5a) edge node [right] {$ $} (7a)
(6a) edge node [right] {$ $} (7a)
(1a) edge node [right] {$ $} (3a)
(2a) edge node [right] {$ $} (3a);
\node (1b) at (2.5,1) {$\bullet$};
\node (2b) at (2.5,0) {$\bullet$};
\node (3b) at (3.25,.8) {$\bullet$};
\node (4b) at (3.25,.2) {$\bullet$};
\node (5b) at (3.85,1) {$\bullet$};
\node (6b) at (3.85,0) {$\bullet$};
\node (7b) at (4.25,.5) {$\bullet$};
\path[font=\small,>=angle 90]
(1b) edge node [right] {$ $} (2b)
(3b) edge node [right] {$ $} (4b)
(3b) edge node [right] {$ $} (5b)
(3b) edge node [right] {$ $} (6b)
(3b) edge node [right] {$ $} (7b)
(4b) edge node [right] {$ $} (5b)
(4b) edge node [right] {$ $} (6b)
(4b) edge node [right] {$ $} (7b)
(5b) edge node [right] {$ $} (6b)
(5b) edge node [right] {$ $} (7b)
(6b) edge node [right] {$ $} (7b)
(1b) edge node [right] {$ $} (3b)
(1b) edge node [right] {$ $} (4b)
(2b) edge node [right] {$ $} (3b)
(2b) edge node [right] {$ $} (4b);
\node (1c) at (5,.8) {$\bullet$};
\node (2c) at (5,.2) {$\bullet$};
\node (3c) at (5.75,.8) {$\bullet$};
\node (4c) at (5.75,.2) {$\bullet$};
\node (5c) at (6.35,1) {$\bullet$};
\node (6c) at (6.35,0) {$\bullet$};
\node (7c) at (6.75,.5) {$\bullet$};
\path[font=\small,>=angle 90]
(1c) edge node [right] {$ $} (2c)
(3c) edge node [right] {$ $} (4c)
(3c) edge node [right] {$ $} (5c)
(3c) edge node [right] {$ $} (6c)
(3c) edge node [right] {$ $} (7c)
(4c) edge node [right] {$ $} (5c)
(4c) edge node [right] {$ $} (6c)
(4c) edge node [right] {$ $} (7c)
(5c) edge node [right] {$ $} (6c)
(5c) edge node [right] {$ $} (7c)
(6c) edge node [right] {$ $} (7c)
(1c) edge node [right] {$ $} (3c)
(1c) edge node [right] {$ $} (4c)
(1c) edge node [right] {$ $} (5c)
(1c) edge node [right] {$ $} (7c)
(2c) edge node [right] {$ $} (6c);
\node (1d) at (0,-.7) {$\bullet$};
\node (2d) at (0,-1.3) {$\bullet$};
\node (3d) at (.75,-.7) {$\bullet$};
\node (4d) at (.75,-1.3) {$\bullet$};
\node (5d) at (1.35,-.5) {$\bullet$};
\node (6d) at (1.35,-1.5) {$\bullet$};
\node (7d) at (1.75,-1) {$\bullet$};
\path[font=\small,>=angle 90]
(1d) edge node [right] {$ $} (2d)
(3d) edge node [right] {$ $} (4d)
(3d) edge node [right] {$ $} (5d)
(3d) edge node [right] {$ $} (6d)
(3d) edge node [right] {$ $} (7d)
(4d) edge node [right] {$ $} (5d)
(4d) edge node [right] {$ $} (6d)
(4d) edge node [right] {$ $} (7d)
(5d) edge node [right] {$ $} (6d)
(5d) edge node [right] {$ $} (7d)
(6d) edge node [right] {$ $} (7d)
(1d) edge node [right] {$ $} (3d)
(1d) edge node [right] {$ $} (5d)
(1d) edge node [right] {$ $} (7d)
(2d) edge node [right] {$ $} (4d)
(2d) edge node [right] {$ $} (6d);
\node (1e) at (2.5,-.7) {$\bullet$};
\node (2e) at (2.5,-1.3) {$\bullet$};
\node (3e) at (3.25,-.7) {$\bullet$};
\node (4e) at (3.25,-1.3) {$\bullet$};
\node (5e) at (3.85,-.5) {$\bullet$};
\node (6e) at (3.85,-1.5) {$\bullet$};
\node (7e) at (4.25,-1) {$\bullet$};
\path[font=\small,>=angle 90]
(1e) edge node [right] {$ $} (2e)
(3e) edge node [right] {$ $} (4e)
(3e) edge node [right] {$ $} (5e)
(3e) edge node [right] {$ $} (6e)
(3e) edge node [right] {$ $} (7e)
(4e) edge node [right] {$ $} (5e)
(4e) edge node [right] {$ $} (6e)
(4e) edge node [right] {$ $} (7e)
(5e) edge node [right] {$ $} (6e)
(5e) edge node [right] {$ $} (7e)
(6e) edge node [right] {$ $} (7e)
(1e) edge node [right] {$ $} (3e)
(1e) edge node [right] {$ $} (4e)
(1e) edge node [right] {$ $} (5e)
(1e) edge node [right] {$ $} (7e)
(2e) edge node [right] {$ $} (4e)
(2e) edge node [right] {$ $} (6e);
\node (1f) at (5,-.7) {$\bullet$};
\node (2f) at (5,-1.3) {$\bullet$};
\node (3f) at (6.75,-.7) {$\bullet$};
\node (4f) at (6.75,-1.3) {$\bullet$};
\node (5f) at (6.15,-.5) {$\bullet$};
\node (6f) at (6.15,-1.5) {$\bullet$};
\node (7f) at (5.75,-1) {$\bullet$};
\path[font=\small,>=angle 90]
(1f) edge node [right] {$ $} (2f)
(3f) edge node [right] {$ $} (4f)
(3f) edge node [right] {$ $} (5f)
(3f) edge node [right] {$ $} (6f)
(3f) edge node [right] {$ $} (7f)
(4f) edge node [right] {$ $} (5f)
(4f) edge node [right] {$ $} (6f)
(4f) edge node [right] {$ $} (7f)
(5f) edge node [right] {$ $} (6f)
(5f) edge node [right] {$ $} (7f)
(6f) edge node [right] {$ $} (7f)
(1f) edge node [right] {$ $} (3f)
(1f) edge node [right] {$ $} (5f)
(1f) edge node [right] {$ $} (7f)
(2f) edge node [right] {$ $} (4f)
(2f) edge node [right] {$ $} (6f)
(2f) edge node [right] {$ $} (7f);
\node (1g) at (0,-2.2) {$\bullet$};
\node (2g) at (0,-2.8) {$\bullet$};
\node (3g) at (.75,-2.5) {$\bullet$};
\node (4g) at (1.15,-2) {$\bullet$};
\node (5g) at (1.15,-3) {$\bullet$};
\node (6g) at (1.75,-2.2) {$\bullet$};
\node (7g) at (1.75,-2.8) {$\bullet$};
\path[font=\small,>=angle 90]
(1g) edge node [right] {$ $} (2g)
(3g) edge node [right] {$ $} (4g)
(3g) edge node [right] {$ $} (5g)
(3g) edge node [right] {$ $} (6g)
(3g) edge node [right] {$ $} (7g)
(4g) edge node [right] {$ $} (5g)
(4g) edge node [right] {$ $} (6g)
(4g) edge node [right] {$ $} (7g)
(5g) edge node [right] {$ $} (6g)
(5g) edge node [right] {$ $} (7g)
(6g) edge node [right] {$ $} (7g)
(1g) edge node [right] {$ $} (3g)
(1g) edge node [right] {$ $} (4g)
(1g) edge node [right] {$ $} (5g)
(2g) edge node [right] {$ $} (3g)
(2g) edge node [right] {$ $} (4g)
(2g) edge node [right] {$ $} (5g);
\node (1h) at (2.5,-2.2) {$\bullet$};
\node (2h) at (2.5,-2.8) {$\bullet$};
\node (3h) at (3.25,-2.2) {$\bullet$};
\node (4h) at (3.25,-2.8) {$\bullet$};
\node (5h) at (3.85,-2) {$\bullet$};
\node (6h) at (3.85,-3) {$\bullet$};
\node (7h) at (4.25,-2.5) {$\bullet$};
\path[font=\small,>=angle 90]
(1h) edge node [right] {$ $} (2h)
(3h) edge node [right] {$ $} (4h)
(3h) edge node [right] {$ $} (5h)
(3h) edge node [right] {$ $} (6h)
(3h) edge node [right] {$ $} (7h)
(4h) edge node [right] {$ $} (5h)
(4h) edge node [right] {$ $} (6h)
(4h) edge node [right] {$ $} (7h)
(5h) edge node [right] {$ $} (6h)
(5h) edge node [right] {$ $} (7h)
(6h) edge node [right] {$ $} (7h)
(1h) edge node [right] {$ $} (3h)
(1h) edge node [right] {$ $} (4h)
(1h) edge node [right] {$ $} (5h)
(1h) edge node [right] {$ $} (7h)
(2h) edge node [right] {$ $} (3h)
(2h) edge node [right] {$ $} (4h)
(2h) edge node [right] {$ $} (6h);
\node (1i) at (5,-2.2) {$\bullet$};
\node (2i) at (5,-2.8) {$\bullet$};
\node (3i) at (5.75,-2.2) {$\bullet$};
\node (4i) at (5.75,-2.8) {$\bullet$};
\node (5i) at (6.35,-2) {$\bullet$};
\node (6i) at (6.35,-3) {$\bullet$};
\node (7i) at (6.75,-2.5) {$\bullet$};
\path[font=\small,>=angle 90]
(1i) edge node [right] {$ $} (2i)
(3i) edge node [right] {$ $} (4i)
(3i) edge node [right] {$ $} (5i)
(3i) edge node [right] {$ $} (6i)
(3i) edge node [right] {$ $} (7i)
(4i) edge node [right] {$ $} (5i)
(4i) edge node [right] {$ $} (6i)
(4i) edge node [right] {$ $} (7i)
(5i) edge node [right] {$ $} (6i)
(5i) edge node [right] {$ $} (7i)
(6i) edge node [right] {$ $} (7i)
(1i) edge node [right] {$ $} (3i)
(1i) edge node [right] {$ $} (4i)
(1i) edge node [right] {$ $} (5i)
(1i) edge node [right] {$ $} (7i)
(2i) edge node [right] {$ $} (3i)
(2i) edge node [right] {$ $} (4i)
(2i) edge node [right] {$ $} (6i)
(2i) edge node [right] {$ $} (7i);
\end{tikzpicture}
$
    \caption{The graphs $B_i$ for $i\in\mathcal{I}$}
    \label{figBDP}
\end{figure}
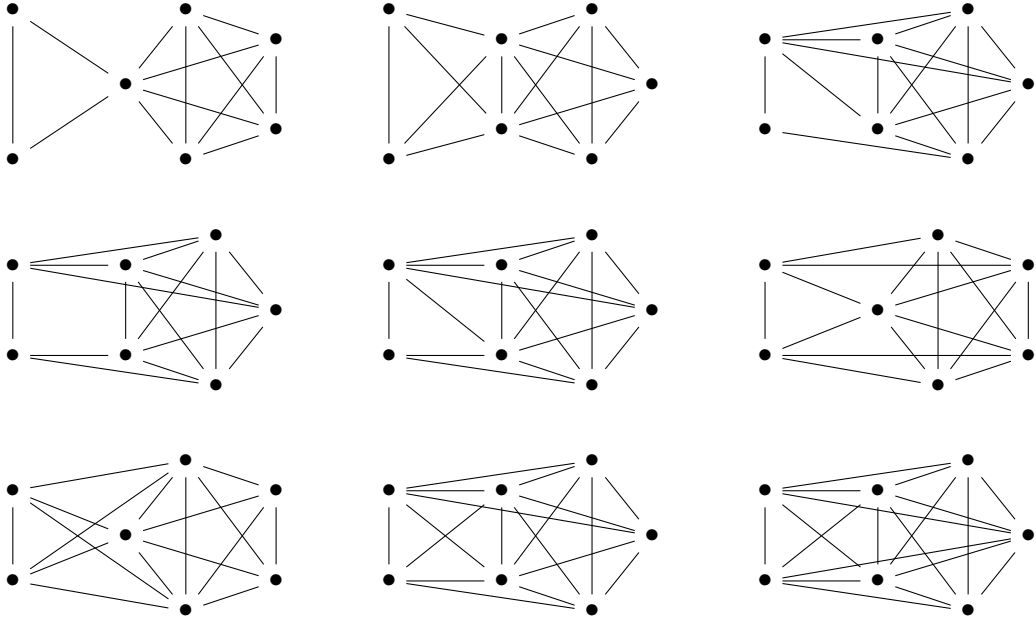

Next we'll consider the graphs in Figure \ref{figBDP}. The graph $B_6$ is again the direct product of the singleton $K_1$ and the disconnected graph with two complete components of sizes two and four (shown to occur in \cite{BLL}). As for $B_{14}$, it can be seen as the direct product between two disconnected graphs: one having complete components of sizes one and one, and the other having complete components of sizes one and four (shown to occur in \cite{L3}). Next, the graph $B_{16}$ is also the direct product between two disconnected graphs known to occur as $\Delta(G)$ for some solvable group $G$; this time it is between one having complete components of sizes one and two, whereas the other has complete components of sizes one and three.

The remaining six graphs in Figure \ref{figBDP} are the direct product between the singleton $K_1$ and one of the graphs shown to occur on page $503$ in \cite{BLL}.

\subsection{Eliminations}\label{secBElim}

In this section, we will eliminate the graphs in Figure \ref{figBDNO}. The first two graphs, $B_1$ and $B_2$, clearly have diameter three, and the last graph, $B_5$, has actually been previously studied.

For $B_1$, notice that it has two cut-vertices. This contradicts the main result from \cite{MQ} which requires prime character degree graphs to have at most one cut vertex. Hence, $B_1$ cannot be the prime character degree graph of any solvable group.

For $B_2$, notice that it has diameter three, as previously mentioned. In particular, following the labeling in Section \ref{secdiam3}, we see that $|\rho_3|=2<3$, which contradicts Proposition \ref{sassresult}\eqref{sass1}. Thus, $B_2$ is not $\Delta(G)$ for any finite solvable group $G$.

Finally, one can see that the graph $B_5$ is exactly the graph $\Gamma_{5,2}$ from \cite{BL}. Using Theorem \ref{KT}, we can conclude that $B_5=\Gamma_{5,2}$ does not occur as the prime character degree graph of any solvable group.

\begin{figure}[htb]
    \centering
$
\begin{tikzpicture}[scale=2]
\node (1a) at (0,1) {$\bullet$};
\node (2a) at (0,0) {$\bullet$};
\node (3a) at (.75,.8) {$\bullet$};
\node (4a) at (.75,.2) {$\bullet$};
\node (5a) at (1.35,1) {$\bullet$};
\node (6a) at (1.35,0) {$\bullet$};
\node (7a) at (1.75,.5) {$\bullet$};
\path[font=\small,>=angle 90]
(1a) edge node [right] {$ $} (2a)
(3a) edge node [right] {$ $} (4a)
(3a) edge node [right] {$ $} (5a)
(3a) edge node [right] {$ $} (6a)
(3a) edge node [right] {$ $} (7a)
(4a) edge node [right] {$ $} (5a)
(4a) edge node [right] {$ $} (6a)
(4a) edge node [right] {$ $} (7a)
(5a) edge node [right] {$ $} (6a)
(5a) edge node [right] {$ $} (7a)
(6a) edge node [right] {$ $} (7a)
(1a) edge node [right] {$ $} (5a);
\node (1b) at (2.5,1) {$\bullet$};
\node (2b) at (2.5,0) {$\bullet$};
\node (3b) at (3.25,.8) {$\bullet$};
\node (4b) at (3.25,.2) {$\bullet$};
\node (5b) at (3.85,1) {$\bullet$};
\node (6b) at (3.85,0) {$\bullet$};
\node (7b) at (4.25,.5) {$\bullet$};
\path[font=\small,>=angle 90]
(1b) edge node [right] {$ $} (2b)
(3b) edge node [right] {$ $} (4b)
(3b) edge node [right] {$ $} (5b)
(3b) edge node [right] {$ $} (6b)
(3b) edge node [right] {$ $} (7b)
(4b) edge node [right] {$ $} (5b)
(4b) edge node [right] {$ $} (6b)
(4b) edge node [right] {$ $} (7b)
(5b) edge node [right] {$ $} (6b)
(5b) edge node [right] {$ $} (7b)
(6b) edge node [right] {$ $} (7b)
(1b) edge node [right] {$ $} (5b)
(1b) edge node [right] {$ $} (3b);
\node (1c) at (5,1) {$\bullet$};
\node (2c) at (5,0) {$\bullet$};
\node (3c) at (5.75,.8) {$\bullet$};
\node (4c) at (5.75,.2) {$\bullet$};
\node (5c) at (6.35,1) {$\bullet$};
\node (6c) at (6.35,0) {$\bullet$};
\node (7c) at (6.75,.5) {$\bullet$};
\path[font=\small,>=angle 90]
(1c) edge node [right] {$ $} (2c)
(3c) edge node [right] {$ $} (4c)
(3c) edge node [right] {$ $} (5c)
(3c) edge node [right] {$ $} (6c)
(3c) edge node [right] {$ $} (7c)
(4c) edge node [right] {$ $} (5c)
(4c) edge node [right] {$ $} (6c)
(4c) edge node [right] {$ $} (7c)
(5c) edge node [right] {$ $} (6c)
(5c) edge node [right] {$ $} (7c)
(6c) edge node [right] {$ $} (7c)
(1c) edge node [right] {$ $} (5c)
(2c) edge node [right] {$ $} (6c);
\end{tikzpicture}
$
    \caption{The graphs $B_1$, $B_2$, and $B_5$}
    \label{figBDNO}
\end{figure}
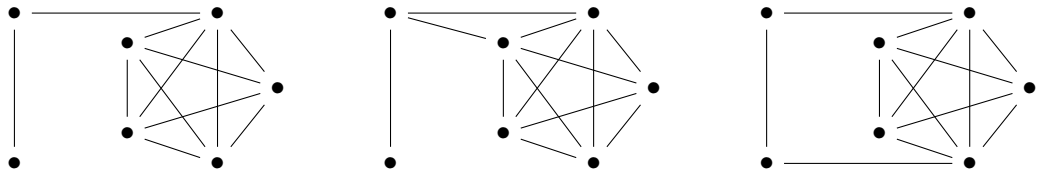

\subsection{Unclassified graphs}\label{secBMaybe}

There are eleven connected graphs that are currently unknown whose vertex sets are covered by cliques of two and five (see Figure \ref{figBM}). Setting
$$
\mathcal{J}:=\{7,8,9,10,11,12,17,18,20,22,25\},
$$
we have the graphs $B_j$ for all $j\in\mathcal{J}$.

\begin{figure}[htb]
    \centering
$
\begin{tikzpicture}[scale=2]
\node (1a) at (0,1) {$\bullet$};
\node (2a) at (0,0) {$\bullet$};
\node (3a) at (.75,.5) {$\bullet$};
\node (4a) at (1.15,1) {$\bullet$};
\node (5a) at (1.15,0) {$\bullet$};
\node (6a) at (1.75,.8) {$\bullet$};
\node (7a) at (1.75,.2) {$\bullet$};
\path[font=\small,>=angle 90]
(1a) edge node [right] {$ $} (2a)
(3a) edge node [right] {$ $} (4a)
(3a) edge node [right] {$ $} (5a)
(3a) edge node [right] {$ $} (6a)
(3a) edge node [right] {$ $} (7a)
(4a) edge node [right] {$ $} (5a)
(4a) edge node [right] {$ $} (6a)
(4a) edge node [right] {$ $} (7a)
(5a) edge node [right] {$ $} (6a)
(5a) edge node [right] {$ $} (7a)
(6a) edge node [right] {$ $} (7a)
(1a) edge node [right] {$ $} (3a)
(1a) edge node [right] {$ $} (4a)
(2a) edge node [right] {$ $} (5a);
\node (1b) at (2.5,1) {$\bullet$};
\node (2b) at (2.5,0) {$\bullet$};
\node (3b) at (3.25,.8) {$\bullet$};
\node (4b) at (3.25,.2) {$\bullet$};
\node (5b) at (3.85,1) {$\bullet$};
\node (6b) at (3.85,0) {$\bullet$};
\node (7b) at (4.25,.5) {$\bullet$};
\path[font=\small,>=angle 90]
(1b) edge node [right] {$ $} (2b)
(3b) edge node [right] {$ $} (4b)
(3b) edge node [right] {$ $} (5b)
(3b) edge node [right] {$ $} (6b)
(3b) edge node [right] {$ $} (7b)
(4b) edge node [right] {$ $} (5b)
(4b) edge node [right] {$ $} (6b)
(4b) edge node [right] {$ $} (7b)
(5b) edge node [right] {$ $} (6b)
(5b) edge node [right] {$ $} (7b)
(6b) edge node [right] {$ $} (7b)
(1b) edge node [right] {$ $} (3b)
(1b) edge node [right] {$ $} (4b)
(2b) edge node [right] {$ $} (4b);
\node (1c) at (5,1) {$\bullet$};
\node (2c) at (5,0) {$\bullet$};
\node (3c) at (5.75,.8) {$\bullet$};
\node (4c) at (5.75,.2) {$\bullet$};
\node (5c) at (6.35,1) {$\bullet$};
\node (6c) at (6.35,0) {$\bullet$};
\node (7c) at (6.75,.5) {$\bullet$};
\path[font=\small,>=angle 90]
(1c) edge node [right] {$ $} (2c)
(3c) edge node [right] {$ $} (4c)
(3c) edge node [right] {$ $} (5c)
(3c) edge node [right] {$ $} (6c)
(3c) edge node [right] {$ $} (7c)
(4c) edge node [right] {$ $} (5c)
(4c) edge node [right] {$ $} (6c)
(4c) edge node [right] {$ $} (7c)
(5c) edge node [right] {$ $} (6c)
(5c) edge node [right] {$ $} (7c)
(6c) edge node [right] {$ $} (7c)
(1c) edge node [right] {$ $} (5c)
(1c) edge node [right] {$ $} (3c)
(1c) edge node [right] {$ $} (4c)
(2c) edge node [right] {$ $} (6c);
\node (1d) at (0,-.5) {$\bullet$};
\node (2d) at (0,-1.5) {$\bullet$};
\node (3d) at (.75,-.7) {$\bullet$};
\node (4d) at (0.75,-1.3) {$\bullet$};
\node (5d) at (1.35,-.5) {$\bullet$};
\node (6d) at (1.35,-1.5) {$\bullet$};
\node (7d) at (1.75,-1) {$\bullet$};
\path[font=\small,>=angle 90]
(1d) edge node [right] {$ $} (2d)
(3d) edge node [right] {$ $} (4d)
(3d) edge node [right] {$ $} (5d)
(3d) edge node [right] {$ $} (6d)
(3d) edge node [right] {$ $} (7d)
(4d) edge node [right] {$ $} (5d)
(4d) edge node [right] {$ $} (6d)
(4d) edge node [right] {$ $} (7d)
(5d) edge node [right] {$ $} (6d)
(5d) edge node [right] {$ $} (7d)
(6d) edge node [right] {$ $} (7d)
(1d) edge node [right] {$ $} (3d)
(1d) edge node [right] {$ $} (4d)
(1d) edge node [right] {$ $} (5d)
(2d) edge node [right] {$ $} (4d);
\node (1e) at (2.5,-.5) {$\bullet$};
\node (2e) at (2.5,-1.5) {$\bullet$};
\node (3e) at (3.25,-.7) {$\bullet$};
\node (4e) at (3.25,-1.3) {$\bullet$};
\node (5e) at (3.85,-.5) {$\bullet$};
\node (6e) at (3.85,-1.5) {$\bullet$};
\node (7e) at (4.25,-1) {$\bullet$};
\path[font=\small,>=angle 90]
(1e) edge node [right] {$ $} (2e)
(3e) edge node [right] {$ $} (4e)
(3e) edge node [right] {$ $} (5e)
(3e) edge node [right] {$ $} (6e)
(3e) edge node [right] {$ $} (7e)
(4e) edge node [right] {$ $} (5e)
(4e) edge node [right] {$ $} (6e)
(4e) edge node [right] {$ $} (7e)
(5e) edge node [right] {$ $} (6e)
(5e) edge node [right] {$ $} (7e)
(6e) edge node [right] {$ $} (7e)
(1e) edge node [right] {$ $} (5e)
(1e) edge node [right] {$ $} (3e)
(2e) edge node [right] {$ $} (6e)
(2e) edge node [right] {$ $} (4e);
\node (1f) at (5,-.5) {$\bullet$};
\node (2f) at (5,-1.5) {$\bullet$};
\node (3f) at (5.75,-1) {$\bullet$};
\node (4f) at (6.15,-.5) {$\bullet$};
\node (5f) at (6.15,-1.5) {$\bullet$};
\node (6f) at (6.75,-.7) {$\bullet$};
\node (7f) at (6.75,-1.3) {$\bullet$};
\path[font=\small,>=angle 90]
(1f) edge node [right] {$ $} (2f)
(3f) edge node [right] {$ $} (4f)
(3f) edge node [right] {$ $} (5f)
(3f) edge node [right] {$ $} (6f)
(3f) edge node [right] {$ $} (7f)
(4f) edge node [right] {$ $} (5f)
(4f) edge node [right] {$ $} (6f)
(4f) edge node [right] {$ $} (7f)
(5f) edge node [right] {$ $} (6f)
(5f) edge node [right] {$ $} (7f)
(6f) edge node [right] {$ $} (7f)
(1f) edge node [right] {$ $} (3f)
(1f) edge node [right] {$ $} (4f)
(2f) edge node [right] {$ $} (3f)
(2f) edge node [right] {$ $} (5f);
\node (1g) at (0,-2) {$\bullet$};
\node (2g) at (0,-3) {$\bullet$};
\node (3g) at (.75,-2.2) {$\bullet$};
\node (4g) at (.75,-2.8) {$\bullet$};
\node (5g) at (1.35,-2) {$\bullet$};
\node (6g) at (1.35,-3) {$\bullet$};
\node (7g) at (1.75,-2.5) {$\bullet$};
\path[font=\small,>=angle 90]
(1g) edge node [right] {$ $} (2g)
(3g) edge node [right] {$ $} (4g)
(3g) edge node [right] {$ $} (5g)
(3g) edge node [right] {$ $} (6g)
(3g) edge node [right] {$ $} (7g)
(4g) edge node [right] {$ $} (5g)
(4g) edge node [right] {$ $} (6g)
(4g) edge node [right] {$ $} (7g)
(5g) edge node [right] {$ $} (6g)
(5g) edge node [right] {$ $} (7g)
(6g) edge node [right] {$ $} (7g)
(1g) edge node [right] {$ $} (3g)
(1g) edge node [right] {$ $} (4g)
(1g) edge node [right] {$ $} (5g)
(2g) edge node [right] {$ $} (4g)
(2g) edge node [right] {$ $} (6g);
\node (1h) at (2.5,-2) {$\bullet$};
\node (2h) at (2.5,-3) {$\bullet$};
\node (3h) at (3.25,-2.2) {$\bullet$};
\node (4h) at (3.25,-2.8) {$\bullet$};
\node (5h) at (3.85,-2) {$\bullet$};
\node (6h) at (3.85,-3) {$\bullet$};
\node (7h) at (4.25,-2.5) {$\bullet$};
\path[font=\small,>=angle 90]
(1h) edge node [right] {$ $} (2h)
(3h) edge node [right] {$ $} (4h)
(3h) edge node [right] {$ $} (5h)
(3h) edge node [right] {$ $} (6h)
(3h) edge node [right] {$ $} (7h)
(4h) edge node [right] {$ $} (5h)
(4h) edge node [right] {$ $} (6h)
(4h) edge node [right] {$ $} (7h)
(5h) edge node [right] {$ $} (6h)
(5h) edge node [right] {$ $} (7h)
(6h) edge node [right] {$ $} (7h)
(1h) edge node [right] {$ $} (3h)
(1h) edge node [right] {$ $} (4h)
(1h) edge node [right] {$ $} (5h)
(2h) edge node [right] {$ $} (3h)
(2h) edge node [right] {$ $} (4h);
\node (1i) at (5,-2) {$\bullet$};
\node (2i) at (5,-3) {$\bullet$};
\node (3i) at (5.75,-2.2) {$\bullet$};
\node (4i) at (5.75,-2.8) {$\bullet$};
\node (5i) at (6.35,-2) {$\bullet$};
\node (6i) at (6.35,-3) {$\bullet$};
\node (7i) at (6.75,-2.5) {$\bullet$};
\path[font=\small,>=angle 90]
(1i) edge node [right] {$ $} (2i)
(3i) edge node [right] {$ $} (4i)
(3i) edge node [right] {$ $} (5i)
(3i) edge node [right] {$ $} (6i)
(3i) edge node [right] {$ $} (7i)
(4i) edge node [right] {$ $} (5i)
(4i) edge node [right] {$ $} (6i)
(4i) edge node [right] {$ $} (7i)
(5i) edge node [right] {$ $} (6i)
(5i) edge node [right] {$ $} (7i)
(6i) edge node [right] {$ $} (7i)
(1i) edge node [right] {$ $} (3i)
(1i) edge node [right] {$ $} (4i)
(1i) edge node [right] {$ $} (5i)
(1i) edge node [right] {$ $} (6i)
(2i) edge node [right] {$ $} (4i)
(2i) edge node [right] {$ $} (6i);
\node (1j) at (1.25,-3.5) {$\bullet$};
\node (2j) at (1.25,-4.5) {$\bullet$};
\node (3j) at (2,-3.7) {$\bullet$};
\node (4j) at (2,-4.3) {$\bullet$};
\node (5j) at (2.6,-3.5) {$\bullet$};
\node (6j) at (2.6,-4.5) {$\bullet$};
\node (7j) at (3,-4) {$\bullet$};
\path[font=\small,>=angle 90]
(1j) edge node [right] {$ $} (2j)
(3j) edge node [right] {$ $} (4j)
(3j) edge node [right] {$ $} (5j)
(3j) edge node [right] {$ $} (6j)
(3j) edge node [right] {$ $} (7j)
(4j) edge node [right] {$ $} (5j)
(4j) edge node [right] {$ $} (6j)
(4j) edge node [right] {$ $} (7j)
(5j) edge node [right] {$ $} (6j)
(5j) edge node [right] {$ $} (7j)
(6j) edge node [right] {$ $} (7j)
(1j) edge node [right] {$ $} (3j)
(1j) edge node [right] {$ $} (4j)
(1j) edge node [right] {$ $} (5j)
(2j) edge node [right] {$ $} (3j)
(2j) edge node [right] {$ $} (4j)
(2j) edge node [right] {$ $} (6j);
\node (1k) at (3.75,-3.5) {$\bullet$};
\node (2k) at (3.75,-4.5) {$\bullet$};
\node (3k) at (4.5,-3.7) {$\bullet$};
\node (4k) at (4.5,-4.3) {$\bullet$};
\node (5k) at (5.1,-3.5) {$\bullet$};
\node (6k) at (5.1,-4.5) {$\bullet$};
\node (7k) at (5.5,-4) {$\bullet$};
\path[font=\small,>=angle 90]
(1k) edge node [right] {$ $} (2k)
(3k) edge node [right] {$ $} (4k)
(3k) edge node [right] {$ $} (5k)
(3k) edge node [right] {$ $} (6k)
(3k) edge node [right] {$ $} (7k)
(4k) edge node [right] {$ $} (5k)
(4k) edge node [right] {$ $} (6k)
(4k) edge node [right] {$ $} (7k)
(5k) edge node [right] {$ $} (6k)
(5k) edge node [right] {$ $} (7k)
(6k) edge node [right] {$ $} (7k)
(1k) edge node [right] {$ $} (3k)
(1k) edge node [right] {$ $} (4k)
(1k) edge node [right] {$ $} (5k)
(1k) edge node [right] {$ $} (6k)
(2k) edge node [right] {$ $} (3k)
(2k) edge node [right] {$ $} (4k)
(2k) edge node [right] {$ $} (6k);
\end{tikzpicture}
$
    \caption{The graphs $B_j$ for $j\in\mathcal{J}$}
    \label{figBM}
\end{figure}
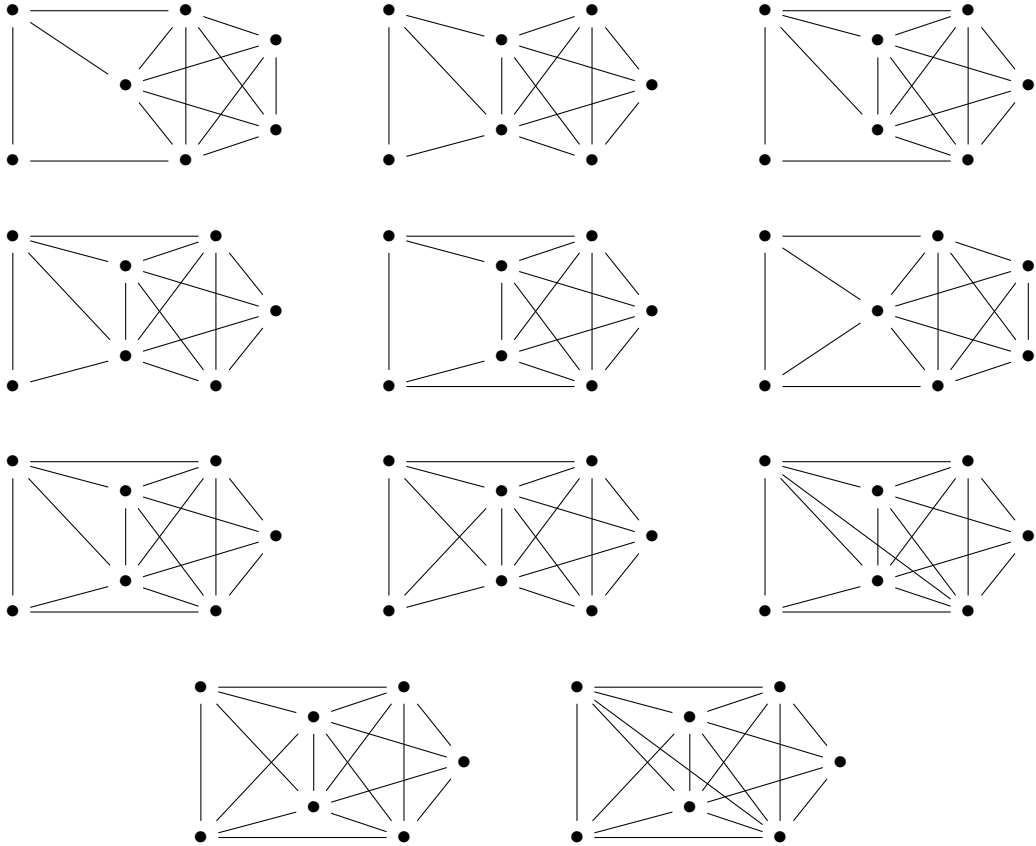

Part of the main reason that these graphs went unclassified is that they rely upon subgraphs with six vertices that are currently unclassified from the work done in \cite{BLL}. The graph $B_7$ is an example of a graph whose classification has been halted by this issue.

The next main issue we faced with these graphs in Figure \ref{figBM} is that they have the disconnected graph with component sizes of two and five that has been shown to occur (see Section \ref{secD2}). This issue can be overcome, but our current argument to eliminate graphs forbids the possibility of a connected subgraph occurring. For example, the graph $B_8$ has $B_6$ as a subgraph, and $B_6$ was shown to occur in Section \ref{secBDP}. We hit a similar issue with the graph $B_9$, for instance, because it has $B_3$ as a subgraph, shown to occur in Section \ref{secBB3}.

\section{Cliques of Three and Four}\label{sec3and4}

There are fifty-three connected graphs whose vertex sets are covered by cliques of three and four (see Appendix \ref{AC}), of which we classify twenty. Of these, we conclude that four of them do indeed occur, whereas sixteen of them do not. The remaining thirty-three are unknown and require further discussion and investigation, along with (presumably) new methods.

Therefore, the goal of this section is to show that all the graphs in Figure \ref{figCDP} occur, and that all the graphs in Figures \ref{figC19} and \ref{figCM} possibly occur. Otherwise, as verified below in detail, the graphs in Figures \ref{figCDNO}, \ref{figCDT}, \ref{figC18}, and \ref{figC20} do not occur as the prime character degree graph of any solvable group.

\subsection{Direct products}\label{secCDP}

There are four graphs in Appendix \ref{AC} that can be constructed as direct products of other prime character degree graphs known to occur. These graphs, shown in Figure \ref{figCDP}, are $C_{26}$, $C_{50}$, $C_{51}$, and $C_{53}$.

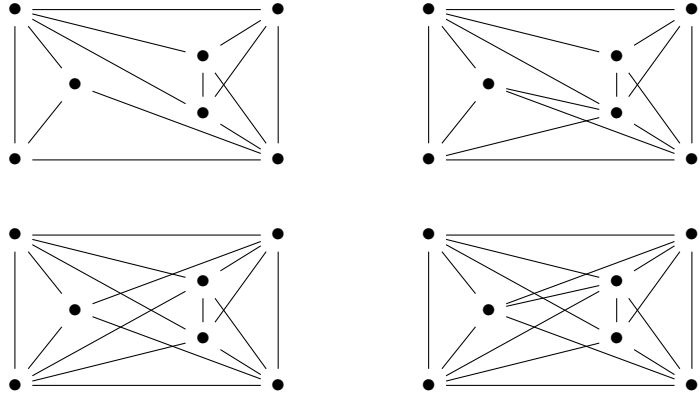
\begin{figure}[htb]
    \centering
$
\begin{tikzpicture}[scale=2]
\node (1a) at (0,1) {$\bullet$};
\node (2a) at (0,0) {$\bullet$};
\node (3a) at (.4,.5) {$\bullet$};
\node (4a) at (1.25,.69) {$\bullet$};
\node (5a) at (1.25,.31) {$\bullet$};
\node (6a) at (1.75,1) {$\bullet$};
\node (7a) at (1.75,0) {$\bullet$};
\path[font=\small,>=angle 90]
(1a) edge node [right] {$ $} (2a)
(1a) edge node [right] {$ $} (3a)
(2a) edge node [right] {$ $} (3a)
(4a) edge node [right] {$ $} (5a)
(4a) edge node [right] {$ $} (6a)
(4a) edge node [right] {$ $} (7a)
(5a) edge node [right] {$ $} (6a)
(5a) edge node [right] {$ $} (7a)
(6a) edge node [right] {$ $} (7a)
(1a) edge node [right] {$ $} (4a)
(1a) edge node [right] {$ $} (5a)
(1a) edge node [right] {$ $} (6a)
(2a) edge node [right] {$ $} (7a)
(3a) edge node [right] {$ $} (7a);
\node (1b) at (2.75,1) {$\bullet$};
\node (2b) at (2.75,0) {$\bullet$};
\node (3b) at (3.15,.5) {$\bullet$};
\node (4b) at (4,.69) {$\bullet$};
\node (5b) at (4,.31) {$\bullet$};
\node (6b) at (4.5,1) {$\bullet$};
\node (7b) at (4.5,0) {$\bullet$};
\path[font=\small,>=angle 90]
(1b) edge node [right] {$ $} (2b)
(1b) edge node [right] {$ $} (3b)
(2b) edge node [right] {$ $} (3b)
(4b) edge node [right] {$ $} (5b)
(4b) edge node [right] {$ $} (6b)
(4b) edge node [right] {$ $} (7b)
(5b) edge node [right] {$ $} (6b)
(5b) edge node [right] {$ $} (7b)
(6b) edge node [right] {$ $} (7b)
(1b) edge node [right] {$ $} (4b)
(1b) edge node [right] {$ $} (5b)
(1b) edge node [right] {$ $} (6b)
(2b) edge node [right] {$ $} (5b)
(2b) edge node [right] {$ $} (7b)
(3b) edge node [right] {$ $} (5b)
(3b) edge node [right] {$ $} (7b);
\node (1c) at (0,-.5) {$\bullet$};
\node (2c) at (0,-1.5) {$\bullet$};
\node (3c) at (.4,-1) {$\bullet$};
\node (4c) at (1.25,-.81) {$\bullet$};
\node (5c) at (1.25,-1.19) {$\bullet$};
\node (6c) at (1.75,-.5) {$\bullet$};
\node (7c) at (1.75,-1.5) {$\bullet$};
\path[font=\small,>=angle 90]
(1c) edge node [right] {$ $} (2c)
(1c) edge node [right] {$ $} (3c)
(2c) edge node [right] {$ $} (3c)
(4c) edge node [right] {$ $} (5c)
(4c) edge node [right] {$ $} (6c)
(4c) edge node [right] {$ $} (7c)
(5c) edge node [right] {$ $} (6c)
(5c) edge node [right] {$ $} (7c)
(6c) edge node [right] {$ $} (7c)
(1c) edge node [right] {$ $} (4c)
(1c) edge node [right] {$ $} (5c)
(1c) edge node [right] {$ $} (6c)
(2c) edge node [right] {$ $} (4c)
(2c) edge node [right] {$ $} (5c)
(2c) edge node [right] {$ $} (7c)
(3c) edge node [right] {$ $} (6c)
(3c) edge node [right] {$ $} (7c);
\node (1d) at (2.75,-.5) {$\bullet$};
\node (2d) at (2.75,-1.5) {$\bullet$};
\node (3d) at (3.15,-1) {$\bullet$};
\node (4d) at (4,-.81) {$\bullet$};
\node (5d) at (4,-1.19) {$\bullet$};
\node (6d) at (4.5,-.5) {$\bullet$};
\node (7d) at (4.5,-1.5) {$\bullet$};
\path[font=\small,>=angle 90]
(1d) edge node [right] {$ $} (2d)
(1d) edge node [right] {$ $} (3d)
(2d) edge node [right] {$ $} (3d)
(4d) edge node [right] {$ $} (5d)
(4d) edge node [right] {$ $} (6d)
(4d) edge node [right] {$ $} (7d)
(5d) edge node [right] {$ $} (6d)
(5d) edge node [right] {$ $} (7d)
(6d) edge node [right] {$ $} (7d)
(1d) edge node [right] {$ $} (4d)
(1d) edge node [right] {$ $} (5d)
(1d) edge node [right] {$ $} (6d)
(2d) edge node [right] {$ $} (4d)
(2d) edge node [right] {$ $} (5d)
(2d) edge node [right] {$ $} (7d)
(3d) edge node [right] {$ $} (4d)
(3d) edge node [right] {$ $} (6d)
(3d) edge node [right] {$ $} (7d);
\end{tikzpicture}
$
    \caption{The graphs $C_{26}$, $C_{50}$, $C_{51}$, and $C_{53}$}
    \label{figCDP}
\end{figure}

For the interested reader, one can see that $C_{26}$ is the direct product between the disconnected graph where one component is $K_2$ and the other $K_3$ and the disconnected graph with two total vertices. Next, the graph $C_{50}$ can be realized as the direct product between the bowtie graph constructed in \cite{L3} and the disconnected graph with two total vertices. The graph $C_{51}$ is again the direct product of the disconnected graph with two total vertices and another graph which occurs from \cite{L3}. Finally, the graph $C_{53}$ is the direct product of the regular graph with six vertices where each vertex is of degree four (see \cite{BLL}, \cite{SST}, or \cite{Z2}) and the singleton. Hence, the graphs in Figure \ref{figCDP} occur as $\Delta(G)$ for some finite solvable group $G$.

\subsection{Previously classified graphs}\label{secCPC}

In this section, we eliminate the graphs in Figure \ref{figCDNO}. Both of these graphs have been previously studied.

\begin{figure}[htb]
    \centering
$
\begin{tikzpicture}[scale=2]
\node (1a) at (0,.5) {$\bullet$};
\node (2a) at (.5,1) {$\bullet$};
\node (3a) at (.5,0) {$\bullet$};
\node (4a) at (1.25,.75) {$\bullet$};
\node (5a) at (1.25,.25) {$\bullet$};
\node (6a) at (1.75,1) {$\bullet$};
\node (7a) at (1.75,0) {$\bullet$};
\path[font=\small,>=angle 90]
(1a) edge node [right] {$ $} (2a)
(1a) edge node [right] {$ $} (3a)
(2a) edge node [right] {$ $} (3a)
(4a) edge node [right] {$ $} (5a)
(4a) edge node [right] {$ $} (6a)
(4a) edge node [right] {$ $} (7a)
(5a) edge node [right] {$ $} (6a)
(5a) edge node [right] {$ $} (7a)
(6a) edge node [right] {$ $} (7a)
(2a) edge node [right] {$ $} (4a)
(2a) edge node [right] {$ $} (6a)
(3a) edge node [right] {$ $} (5a)
(3a) edge node [right] {$ $} (7a);
\node (1b) at (2.75,1) {$\bullet$};
\node (2b) at (2.75,0) {$\bullet$};
\node (3b) at (3.25,.5) {$\bullet$};
\node (4b) at (3.75,.5) {$\bullet$};
\node (5b) at (4.125,1) {$\bullet$};
\node (6b) at (4.125,0) {$\bullet$};
\node (7b) at (4.5,.5) {$\bullet$};
\path[font=\small,>=angle 90]
(1b) edge node [right] {$ $} (2b)
(1b) edge node [right] {$ $} (3b)
(2b) edge node [right] {$ $} (3b)
(4b) edge node [right] {$ $} (5b)
(4b) edge node [right] {$ $} (6b)
(4b) edge node [right] {$ $} (7b)
(5b) edge node [right] {$ $} (6b)
(5b) edge node [right] {$ $} (7b)
(6b) edge node [right] {$ $} (7b)
(1b) edge node [right] {$ $} (5b)
(2b) edge node [right] {$ $} (6b)
(3b) edge node [right] {$ $} (4b);
\end{tikzpicture}
$
    \caption{The graphs $C_{10}$ and $C_{17}$}
    \label{figCDNO}
\end{figure}
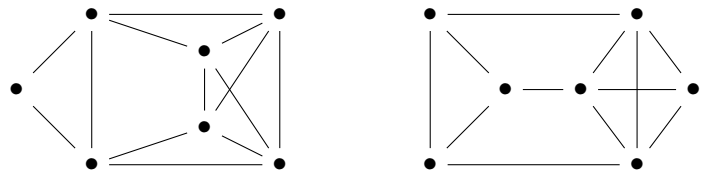

The first graph, $C_{10}$, was studied in \cite{BLL}. It was also classified in a family that was studied in \cite{LM}, and in that context was denoted by $\Sigma_{2,2}^L$. Following either result, we can conclude that $\Sigma_{2,2}^L=C_{10}$ does not occur as $\Delta(G)$ for any solvable group $G$.

Furthermore, the second graph, $C_{17}$, is exactly the graph $\Gamma_{4,3}$, which was classified in \cite{BL}. In that paper, the authors concluded that $\Gamma_{4,3}=C_{17}$ also does not occur as the prime character degree graph of any solvable group (see Theorem \ref{KT}).

\subsection{Diameter three graphs}\label{secCDT}

There are twelve graphs that are covered by cliques of three and four that have diameter three. Specifically, the graph $C_i$ has diameter three for all $i\in\{1,2,3,4,5,6,7,9,11,12,14,16\}$. For ease of notation, define
$$
\mathcal{I}:=\{1,2,3,4,5,6,7,9,11,12,14,16\}.
$$
One can see Figure \ref{figCDT} for the complete list. We claim that $C_i$ does not occur as the prime character degree graph for any solvable group for all $i\in\mathcal{I}$.

\begin{figure}[htb]
    \centering
$
\begin{tikzpicture}[scale=2]
\node (1a) at (0,1) {$\bullet$};
\node (2a) at (0,0) {$\bullet$};
\node (3a) at (.5,.5) {$\bullet$};
\node (4a) at (1,.5) {$\bullet$};
\node (5a) at (1.375,1) {$\bullet$};
\node (6a) at (1.375,0) {$\bullet$};
\node (7a) at (1.75,.5) {$\bullet$};
\path[font=\small,>=angle 90]
(1a) edge node [right] {$ $} (2a)
(1a) edge node [right] {$ $} (3a)
(2a) edge node [right] {$ $} (3a)
(4a) edge node [right] {$ $} (5a)
(4a) edge node [right] {$ $} (6a)
(4a) edge node [right] {$ $} (7a)
(5a) edge node [right] {$ $} (6a)
(5a) edge node [right] {$ $} (7a)
(6a) edge node [right] {$ $} (7a)
(3a) edge node [right] {$ $} (4a);
\node (1b) at (2.5,1) {$\bullet$};
\node (2b) at (2.5,0) {$\bullet$};
\node (3b) at (3,.5) {$\bullet$};
\node (4b) at (3.75,.69) {$\bullet$};
\node (5b) at (3.75,.31) {$\bullet$};
\node (6b) at (4.25,1) {$\bullet$};
\node (7b) at (4.25,0) {$\bullet$};
\path[font=\small,>=angle 90]
(1b) edge node [right] {$ $} (2b)
(1b) edge node [right] {$ $} (3b)
(2b) edge node [right] {$ $} (3b)
(4b) edge node [right] {$ $} (5b)
(4b) edge node [right] {$ $} (6b)
(4b) edge node [right] {$ $} (7b)
(5b) edge node [right] {$ $} (6b)
(5b) edge node [right] {$ $} (7b)
(6b) edge node [right] {$ $} (7b)
(3b) edge node [right] {$ $} (4b)
(3b) edge node [right] {$ $} (5b);
\node (1c) at (5,1) {$\bullet$};
\node (2c) at (5,0) {$\bullet$};
\node (3c) at (5.5,.5) {$\bullet$};
\node (4c) at (6,.5) {$\bullet$};
\node (5c) at (6.375,1) {$\bullet$};
\node (6c) at (6.375,0) {$\bullet$};
\node (7c) at (6.75,.5) {$\bullet$};
\path[font=\small,>=angle 90]
(1c) edge node [right] {$ $} (2c)
(1c) edge node [right] {$ $} (3c)
(2c) edge node [right] {$ $} (3c)
(4c) edge node [right] {$ $} (5c)
(4c) edge node [right] {$ $} (6c)
(4c) edge node [right] {$ $} (7c)
(5c) edge node [right] {$ $} (6c)
(5c) edge node [right] {$ $} (7c)
(6c) edge node [right] {$ $} (7c)
(3c) edge node [right] {$ $} (4c)
(3c) edge node [right] {$ $} (5c)
(3c) edge node [right] {$ $} (6c);
\node (1d) at (0,-1) {$\bullet$};
\node (2d) at (.5,-.5) {$\bullet$};
\node (3d) at (.5,-1.5) {$\bullet$};
\node (4d) at (1.25,-.81) {$\bullet$};
\node (5d) at (1.25,-1.29) {$\bullet$};
\node (6d) at (1.75,-.5) {$\bullet$};
\node (7d) at (1.75,-1.5) {$\bullet$};
\path[font=\small,>=angle 90]
(1d) edge node [right] {$ $} (2d)
(1d) edge node [right] {$ $} (3d)
(2d) edge node [right] {$ $} (3d)
(4d) edge node [right] {$ $} (5d)
(4d) edge node [right] {$ $} (6d)
(4d) edge node [right] {$ $} (7d)
(5d) edge node [right] {$ $} (6d)
(5d) edge node [right] {$ $} (7d)
(6d) edge node [right] {$ $} (7d)
(2d) edge node [right] {$ $} (4d)
(3d) edge node [right] {$ $} (5d);
\node (1e) at (2.5,-1) {$\bullet$};
\node (2e) at (3,-.5) {$\bullet$};
\node (3e) at (3,-1.5) {$\bullet$};
\node (4e) at (3.5,-1) {$\bullet$};
\node (5e) at (3.875,-.5) {$\bullet$};
\node (6e) at (3.875,-1.5) {$\bullet$};
\node (7e) at (4.25,-1) {$\bullet$};
\path[font=\small,>=angle 90]
(1e) edge node [right] {$ $} (2e)
(1e) edge node [right] {$ $} (3e)
(2e) edge node [right] {$ $} (3e)
(4e) edge node [right] {$ $} (5e)
(4e) edge node [right] {$ $} (6e)
(4e) edge node [right] {$ $} (7e)
(5e) edge node [right] {$ $} (6e)
(5e) edge node [right] {$ $} (7e)
(6e) edge node [right] {$ $} (7e)
(2e) edge node [right] {$ $} (4e)
(3e) edge node [right] {$ $} (4e);
\node (1f) at (5,-1) {$\bullet$};
\node (2f) at (5.5,-.5) {$\bullet$};
\node (3f) at (5.5,-1.5) {$\bullet$};
\node (4f) at (6,-1) {$\bullet$};
\node (5f) at (6.375,-.5) {$\bullet$};
\node (6f) at (6.375,-1.5) {$\bullet$};
\node (7f) at (6.75,-1) {$\bullet$};
\path[font=\small,>=angle 90]
(1f) edge node [right] {$ $} (2f)
(1f) edge node [right] {$ $} (3f)
(2f) edge node [right] {$ $} (3f)
(4f) edge node [right] {$ $} (5f)
(4f) edge node [right] {$ $} (6f)
(4f) edge node [right] {$ $} (7f)
(5f) edge node [right] {$ $} (6f)
(5f) edge node [right] {$ $} (7f)
(6f) edge node [right] {$ $} (7f)
(2f) edge node [right] {$ $} (4f)
(2f) edge node [right] {$ $} (5f)
(3f) edge node [right] {$ $} (6f);
\node (1g) at (0,-2.5) {$\bullet$};
\node (2g) at (.5,-2) {$\bullet$};
\node (3g) at (.5,-3) {$\bullet$};
\node (4g) at (1.25,-2.31) {$\bullet$};
\node (5g) at (1.25,-2.69) {$\bullet$};
\node (6g) at (1.75,-2) {$\bullet$};
\node (7g) at (1.75,-3) {$\bullet$};
\path[font=\small,>=angle 90]
(1g) edge node [right] {$ $} (2g)
(1g) edge node [right] {$ $} (3g)
(2g) edge node [right] {$ $} (3g)
(4g) edge node [right] {$ $} (5g)
(4g) edge node [right] {$ $} (6g)
(4g) edge node [right] {$ $} (7g)
(5g) edge node [right] {$ $} (6g)
(5g) edge node [right] {$ $} (7g)
(6g) edge node [right] {$ $} (7g)
(2g) edge node [right] {$ $} (4g)
(2g) edge node [right] {$ $} (5g)
(3g) edge node [right] {$ $} (5g);
\node (1h) at (2.5,-2.5) {$\bullet$};
\node (2h) at (3,-2) {$\bullet$};
\node (3h) at (3,-3) {$\bullet$};
\node (4h) at (3.75,-2.31) {$\bullet$};
\node (5h) at (3.75,-2.69) {$\bullet$};
\node (6h) at (4.25,-2) {$\bullet$};
\node (7h) at (4.25,-3) {$\bullet$};
\path[font=\small,>=angle 90]
(1h) edge node [right] {$ $} (2h)
(1h) edge node [right] {$ $} (3h)
(2h) edge node [right] {$ $} (3h)
(4h) edge node [right] {$ $} (5h)
(4h) edge node [right] {$ $} (6h)
(4h) edge node [right] {$ $} (7h)
(5h) edge node [right] {$ $} (6h)
(5h) edge node [right] {$ $} (7h)
(6h) edge node [right] {$ $} (7h)
(2h) edge node [right] {$ $} (4h)
(2h) edge node [right] {$ $} (5h)
(2h) edge node [right] {$ $} (6h)
(3h) edge node [right] {$ $} (5h);
\node (1i) at (5,-2.5) {$\bullet$};
\node (2i) at (5.5,-2) {$\bullet$};
\node (3i) at (5.5,-3) {$\bullet$};
\node (4i) at (6,-2.5) {$\bullet$};
\node (5i) at (6.375,-2) {$\bullet$};
\node (6i) at (6.375,-3) {$\bullet$};
\node (7i) at (6.75,-2.5) {$\bullet$};
\path[font=\small,>=angle 90]
(1i) edge node [right] {$ $} (2i)
(1i) edge node [right] {$ $} (3i)
(2i) edge node [right] {$ $} (3i)
(4i) edge node [right] {$ $} (5i)
(4i) edge node [right] {$ $} (6i)
(4i) edge node [right] {$ $} (7i)
(5i) edge node [right] {$ $} (6i)
(5i) edge node [right] {$ $} (7i)
(6i) edge node [right] {$ $} (7i)
(2i) edge node [right] {$ $} (4i)
(2i) edge node [right] {$ $} (5i)
(3i) edge node [right] {$ $} (4i)
(3i) edge node [right] {$ $} (6i);
\node (1j) at (0,-4) {$\bullet$};
\node (2j) at (.5,-3.5) {$\bullet$};
\node (3j) at (.5,-4.5) {$\bullet$};
\node (4j) at (1.25,-3.81) {$\bullet$};
\node (5j) at (1.25,-4.19) {$\bullet$};
\node (6j) at (1.75,-3.5) {$\bullet$};
\node (7j) at (1.75,-4.5) {$\bullet$};
\path[font=\small,>=angle 90]
(1j) edge node [right] {$ $} (2j)
(1j) edge node [right] {$ $} (3j)
(2j) edge node [right] {$ $} (3j)
(4j) edge node [right] {$ $} (5j)
(4j) edge node [right] {$ $} (6j)
(4j) edge node [right] {$ $} (7j)
(5j) edge node [right] {$ $} (6j)
(5j) edge node [right] {$ $} (7j)
(6j) edge node [right] {$ $} (7j)
(2j) edge node [right] {$ $} (4j)
(2j) edge node [right] {$ $} (5j)
(3j) edge node [right] {$ $} (4j)
(3j) edge node [right] {$ $} (5j);
\node (1k) at (2.5,-4) {$\bullet$};
\node (2k) at (3,-3.5) {$\bullet$};
\node (3k) at (3,-4.5) {$\bullet$};
\node (4k) at (3.75,-3.81) {$\bullet$};
\node (5k) at (3.75,-4.19) {$\bullet$};
\node (6k) at (4.25,-3.5) {$\bullet$};
\node (7k) at (4.25,-4.5) {$\bullet$};
\path[font=\small,>=angle 90]
(1k) edge node [right] {$ $} (2k)
(1k) edge node [right] {$ $} (3k)
(2k) edge node [right] {$ $} (3k)
(4k) edge node [right] {$ $} (5k)
(4k) edge node [right] {$ $} (6k)
(4k) edge node [right] {$ $} (7k)
(5k) edge node [right] {$ $} (6k)
(5k) edge node [right] {$ $} (7k)
(6k) edge node [right] {$ $} (7k)
(2k) edge node [right] {$ $} (4k)
(2k) edge node [right] {$ $} (5k)
(2k) edge node [right] {$ $} (6k)
(3k) edge node [right] {$ $} (4k)
(3k) edge node [right] {$ $} (5k);
\node (1l) at (5,-4) {$\bullet$};
\node (2l) at (5.375,-3.5) {$\bullet$};
\node (3l) at (5.375,-4.5) {$\bullet$};
\node (4l) at (6.175,-4) {$\bullet$};
\node (5l) at (6.375,-3.5) {$\bullet$};
\node (6l) at (6.375,-4.5) {$\bullet$};
\node (7l) at (6.75,-4) {$\bullet$};
\path[font=\small,>=angle 90]
(1l) edge node [right] {$ $} (2l)
(1l) edge node [right] {$ $} (3l)
(2l) edge node [right] {$ $} (3l)
(4l) edge node [right] {$ $} (5l)
(4l) edge node [right] {$ $} (6l)
(4l) edge node [right] {$ $} (7l)
(5l) edge node [right] {$ $} (6l)
(5l) edge node [right] {$ $} (7l)
(6l) edge node [right] {$ $} (7l)
(2l) edge node [right] {$ $} (4l)
(2l) edge node [right] {$ $} (5l)
(2l) edge node [right] {$ $} (6l)
(3l) edge node [right] {$ $} (4l)
(3l) edge node [right] {$ $} (5l)
(3l) edge node [right] {$ $} (6l);
\end{tikzpicture}
$
    \caption{The graphs $C_i$ for $i\in\mathcal{I}$}
    \label{figCDT}
\end{figure}
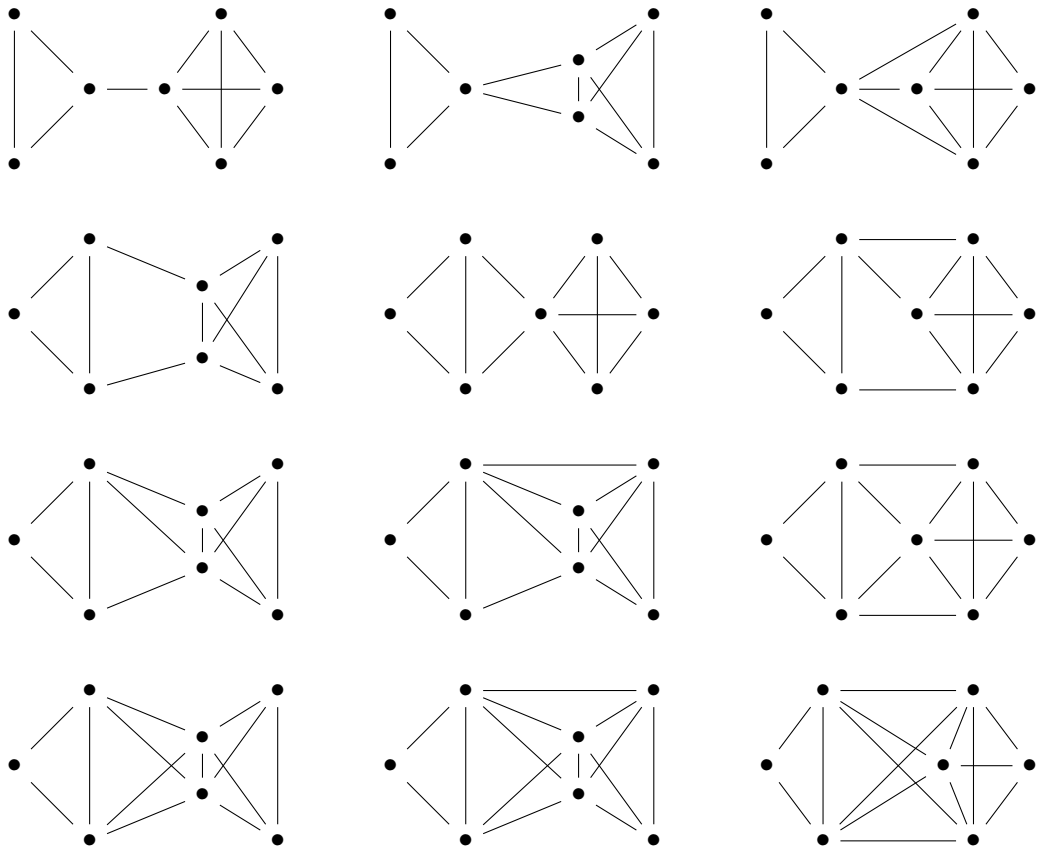

One can use the main result in \cite{MQ} to eliminate the graph $C_1$, because that graph has two cut-vertices. However, a simple method to eliminate $C_i$ for all $i\in\mathcal{I}$ is to employ the results in \cite{S} (see Proposition \ref{sassresult}\eqref{sass3}). Using the prescribed notation set in Section \ref{secdiam3}, one can see for each $i\in\mathcal{I}$ we have for $C_i$ that $|\rho_1\cup\rho_2|=3$ and $|\rho_3\cup\rho_4|=4$, yet
$$
4=|\rho_3\cup\rho_4|\geq2^{|\rho_1\cup\rho_2|}=2^3=8,
$$
a contradiction.

\subsection{Eliminating the graph $C_{18}$}\label{secC18}

Our goal in this section is to eliminate the graph $C_{18}$ (see Figure \ref{figC18}). In other words, we will conclude that $C_{18}$ does not occur as the prime character degree graph for any solvable group. We follow a similar strategy that is employed in \cite{L3}, or even \cite{BLL}, \cite{DLM}, or \cite{LM}.

\begin{figure}[htb]
    \centering
$
\begin{tikzpicture}[scale=2.5]
\node (1) at (0,1) {$p_1$};
\node (2) at (0,0) {$p_2$};
\node (3) at (0.5,.5) {$p_3$};
\node (4) at (1.25, .75) {$p_4$};
\node (5) at (1.25, .25) {$p_5$};
\node (6) at (1.75, 1) {$p_6$};
\node (7) at (1.75, 0) {$p_7$};
\path[font=\small,>=angle 90]
(1) edge node [right] {$ $} (2)
(1) edge node [above] {$ $} (3)
(2) edge node [right] {$ $} (3)
(4) edge node [right] {$ $} (5)
(4) edge node [right] {$ $} (6)
(4) edge node [right] {$ $} (7)
(5) edge node [right] {$ $} (6)
(5) edge node [right] {$ $} (7)
(6) edge node [right] {$ $} (7)
(1) edge node [right] {$ $} (4)
(2) edge node [right] {$ $} (5)
(3) edge node [right] {$ $} (4);
\end{tikzpicture}
$
    \caption{The graph $C_{18}$}
    \label{figC18}
\end{figure}
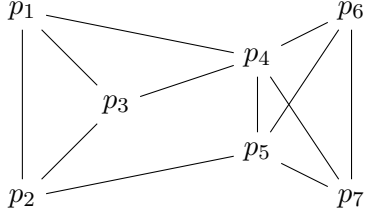

\begin{lemma}\label{ad1}
Assume $C_{18}=\Delta(G)$ for some finite solvable group $G$, where $|G|$ is minimal. Then $G$ does not have a normal nonabelian Sylow $p_i$-subgroup for $i\in\{1,3,4,5,6,7\}$. In particular, each $p_i$ is a strongly admissible vertex.
\end{lemma}
\begin{proof}
We begin by checking for admissibility of the vertex $p_1$. If we remove the vertex $p_1$ and all incident edges, we are left with the graph $C_{18}[p_1]=\Gamma_{4,2}$, which is non-occurring from \cite{BL} (see Theorem \ref{KT}). Next, the edges $\epsilon(p_1,p_2)$ and $\epsilon(p_1,p_3)$ must remain, for losing either or both would violate P\'{a}lfy's condition. Additionally, removing the edge between $p_1$ and $p_4$ leaves us with a diameter three graph $C_{18}[\epsilon(p_1,p_4)]=C_{4}$, which violates the main result from \cite{S} and has been shown to not occur in Section \ref{secCDT}. Therefore, $p_1$ is admissible.

For strong admissibility, we consider the removal of $p_1$, all incident edges, and one or more of the edges between two vertices adjacent to $p_1$. It is easy to see that $C_{18}[p_1, \epsilon(p_2, p_3)]$ violates P\'{a}lfy's condition, and so we cannot lose the edge between $p_2$ and $p_3$. Now, $C_{18}[p_1, \epsilon(p_3, p_4)]$ has diameter three, which again does not occur by \cite{S}. Thus, $p_1$ is strongly admissible. By symmetry, we also have that $p_3$ is strongly admissible.

Next, we consider the vertex $p_4$. The graph $C_{18}[p_4]$ has diameter three and is not the unique six-vertex diameter three graph which occurs in \cite{L2}, so it is therefore non-occurring. The graph $C_{18}[\epsilon(p_4, p_5)]$ contains a five-cycle in its complement, which means it violates the generalized P\'{a}lfy's condition from \cite{A}. The vertices $p_6$ and $p_7$ are symmetric, and thus the edges $\epsilon(p_4, p_6)$ and $\epsilon(p_4, p_7)$ are symmetric. Losing either of the aforementioned edges would result in violations of P\'{a}lfy's condition. It has already been established that we cannot remove the edge $\epsilon(p_1, p_4)$. Symmetry demands that we also cannot remove the edge $\epsilon(p_3, p_4)$. Thus, if we lose any combination of edges from $p_4$, the resulting graphs will all be non-occurring due to either P\'{a}lfy's condition or having diameter three. Hence, $p_4$ is admissible.

For strong admissibility, we consider the removal of $p_4$, all incident edges, and one or more of the edges between two vertices adjacent to $p_4$. Observe that $C_{18}[p_4]$ has diameter three. From here, removing any number of edges between two vertices adjacent to $p_4$ will either maintain our diameter three violation or create a non-occurring disconnected graph with six vertices classified in \cite{BLL}. Therefore, $p_4$ is strongly admissible.

Consider the vertex $p_5$. Notice that $C_{18}[p_5]$ has diameter three and is non-occurring by \cite{S}. Next, we established previously that we cannot lose $\epsilon(p_4, p_5)$. Similarly, the symmetry between $p_6$ and $p_7$ makes $\epsilon(p_5, p_6)$ and $\epsilon(p_5, p_7)$ symmetric as well, and the removal of either edge would give us a graph which violates P\'{a}lfy's condition, so these graphs again do not occur. Removing the edge between $p_2$ and $p_5$ creates the diameter three graph $C_{18}[\epsilon(p_2,p_5)]=C_5$, which violates the main result from \cite{S} and has also been shown not to occur in Section \ref{secCDT}. Thus, losing any combination of edges incident to $p_5$ results in a graph that does not occur. Therefore, $p_5$ is admissible.

For strong admissibility, we consider the removal of $p_5$, all incident edges, and one or more of the edges between two vertices adjacent to $p_5$. Notice that the graph $C_{18}[p_5]$ has diameter three and does not occur. All edges between two vertices adjacent to $p_5$ lie in the complete subgraph on vertices $p_4$, $p_6$, and $p_7$, and so the graphs obtained by removing any one or more of these edges will create either a diameter three graph or a non-occurring disconnected graph on six vertices. Thus, $p_5$ is strongly admissible. 

We now check for admissibility of $p_6$. The graph $C_{18}[p_6]$ is simply the graph $\Sigma^{L}_{2,1}$ which was shown not to occur in \cite{LM}. It is easy to see that removing one or more of the edges incident to $p_6$ will violate P\'{a}lfy's condition. Thus, $p_6$ is admissible. 

For strong admissibility, we consider the removal of $p_6$, all incident edges, and one or more of the edges between two vertices adjacent to $p_6$. The graph $C_{18}[p_6, \epsilon(p_4, p_5)]$ has a five-cycle in its complement, thus violating the generalized P\'{a}lfy's condition from \cite{A}. Similarly, removing the edges $\epsilon(p_4, p_7)$ or $\epsilon(p_5, p_7)$ will also violate this generalized P\'{a}lfy's condition. Hence, $p_6$ is strongly admissible. Again by symmetry, we also have that $p_7$ is strongly admissible.

Since $p_i$ is strongly admissible for each $i\in\{1,3,4,5,6,7\}$, then there is no normal nonabelian Sylow $p_i$-subgroup for each $i$ by way of Lemma \ref{strong}.
\end{proof}

\begin{lemma}\label{pi1}
Assume $C_{18}=\Delta(G)$ for some finite solvable group $G$, where $|G|$ is minimal. There exists no normal nonabelian Sylow $p_2$-subgroup.
\end{lemma}
\begin{proof}
We will follow the notation from Lemma \ref{pi}. For $q=p_2$, we have $\pi=\{p_1, p_3, p_5\}$ and $\rho= \{p_4, p_6, p_7\}$. Setting $\pi_1 = \{p_5\}$ and $\pi_2 = \{p_1, p_3\}$, we choose $v = p_1$, which is adjacent to $s = p_4\in\rho$, shown to be admissible in Lemma \ref{ad1}. We also note that $w = p_6 \in \rho$ is not adjacent to $v$, and hence the result follows from Lemma \ref{pi}.
\end{proof}

\begin{lemma}\label{sub1}
No proper subgraph with the same vertex set as $C_{18}$ occurs as the prime character degree graph of any solvable group.
\end{lemma}
\begin{proof}
We have established already that we cannot lose any edges in the complete graphs on $\{p_1, p_2, p_3\}$ and $\{p_4, p_5, p_6, p_7\}$. Thus, we are left to consider the three edges $\epsilon(p_1, p_4)$, $\epsilon(p_3, p_4)$, and $\epsilon(p_2, p_5)$. We will check all subgraphs for which we remove one or more of those edges.

First, recall that $p_1$ and $p_3$ are symmetric. This means that $\epsilon(p_1, p_4)$ and $\epsilon(p_3, p_4)$ are symmetric as well. Notice that $C_{18}[\epsilon(p_1, p_4)] = C_{18}[\epsilon(p_3, p_4)] = C_4$ and $C_{18}[\epsilon(p_2, p_5)] = C_5$ have both been shown not to occur. Next, considering losing two of the aforementioned three edges, we have $C_{18}[\epsilon(p_1,p_4),\epsilon(p_3,p_4)]=C_{18}[\epsilon(p_1,p_4),\epsilon(p_2,p_5)]=C_{18}[\epsilon(p_3,p_4),\epsilon(p_2,p_5)]=C_1$, which does not occur (shown in Section \ref{secCDT}).

Lastly, removing all three edges will give us the disconnected graph in Figure \ref{figDDNO}, which was shown not to occur in Section \ref{secD3}.

Hence, no proper subgraph of $C_{18}$ with the same vertex set occurs as the prime character degree graph of any solvable group.
\end{proof}

Assuming that there exists some solvable group $G$ with $|G|$ minimal such that $\Delta(G)=C_{18}$, we now have the conclusion that there are no normal nonabelian Sylow $p$-subgroups for all $p\in\rho(G)$ (using Lemmas \ref{ad1} and \ref{pi1}). With $F$ as the Fitting subgroup of $G$, we then note that $\rho(G)=\pi(|G:F|)$. Therefore, $\rho(G)=\rho(G/\Phi(G))$. Here, as is convention, $\Phi(G)$ denotes the Frattini subgroup of $G$. Lemma \ref{sub1} above verifies there is no proper subgraph of $C_{18}$ with the same vertex set which occurs as the prime character degree graph of any solvable group, and therefore, since $|G|$ is minimal, we are forced to conclude that $\Phi(G)=1$. We can then apply Lemma III 4.4 from \cite{H} to obtain the existence of a subgroup $H$ of $G$ such that $G=HF$ and $H\cap F=1$.

\begin{lemma}\label{min1}
Assume $C_{18}=\Delta(G)$ for some finite solvable group $G$, where $|G|$ is minimal. Then the Fitting subgroup $F$ of $G$ is minimal normal in $G$.
\end{lemma}
\begin{proof}
Following the assumptions, we then can define $H$ as above. Furthermore, set $E$ to be the Fitting subgroup of $H$.

Next, we proceed by contradiction. Thus, we suppose that there exists some normal subgroup $N$ of $G$ such that $1<N<F$. We aim to show that no such $N$ can exist.

Notice again by \cite{H}, this time III 4.5, that we know there exists some normal subgroup $M$ of $G$ such that $F=N\times M$, and since both $N$ and $M$ are nontrivial, we have that $\rho(G/N)\subset\rho(G)$ and $\rho(G/M)\subset\rho(G)$. For any prime $q\in\rho(G)\setminus\rho(G/N)$, it is known that $G/N$ must have a normal nonabelian Sylow $q$-subgroup whose class is a formation. Therefore, it is required that $q\in\rho(G/M)$. Hence, we get that $\rho(G)=\rho(G/N)\cup\rho(G/M)$.

So for this $q\in\rho(G)\setminus\rho(G/N)$, notice that $q\not\in\rho(G/F)=\rho(H)$, and so $E$ must contain the Sylow $q$-subgroup of $H$. Since $q\in\rho(G)$, however, we must have that $q$ divides $|H|$, and so $q$ divides $|E|$ too. Then since $\cd(G)$ contains a degree divisible by all the prime divisors of $|EF:F|=|E|$, we get that $\rho(G)\setminus(\rho(G/N)\cap\rho(G/M))$ lies in a complete subgraph of $\Delta(G)$. Hence, $\rho(G)\setminus(\rho(G/N)\cap\rho(G/M))$ must lie in one of the following sets: (a) $\{p_1, p_2, p_3\}$, (b) $\{p_1, p_3, p_4\}$, (c) $\{p_2, p_5\}$, or (d) $\{p_4, p_5, p_6, p_7\}$.

Suppose (a) occurs; that is, suppose $\rho(G)\setminus(\rho(G/N)\cap\rho(G/M)) \subseteq \{p_1, p_2, p_3\}$. Following \cite{L3}, we see that $E$ has a Hall $\{p_1,p_2,p_3\}$-subgroup of $H$. One can then conclude that $E$ is in fact the Hall $\{p_1,p_2,p_3\}$-subgroup since $\cd(G)$ has a degree that is divisible by all the primes dividing $|E|$ and $|E|$ is divisible by no other primes. Next, we can find a character $\chi\in\Irr(G)$ with $p_1p_2p_3$ dividing $\chi(1)$. Letting $\theta$ be an irreducible constituent of $\chi_{FE}$, one can see that $\chi(1)/\theta(1)$ divides $|G:FE|$ and that $\chi(1)$ is relatively prime to $|G:FE|$. This forces $\chi_{FE}=\theta$. Since $p_1$, $p_2$, and $p_3$ all divide $\theta(1)$, and the only possible primes that could divide a character in $\cd(G/FE)$ are $p_4$, $p_5$, $p_6$, or $p_7$, we get that $\cd(G/FE)=\{1\}$ by way of Gallagher's Theorem. Hence, $G/FE$ is abelian. For instance, this then yields $O^{p_4}(G)<G$, a contradiction as $p_4$ is admissible.

Suppose (b) occurs; that is, suppose $\rho(G)\setminus(\rho(G/N)\cap\rho(G/M)) \subseteq \{p_1, p_3, p_4\}$. This implies that $\{p_2, p_5, p_6, p_7\} \subseteq \rho(G/N)\cap\rho(G/M)$. Since $p_1\in\rho(G)=\rho(G/N)\cup\rho(G/M)$, then we know that $p_1\in\rho(G/N)$ or $p_1\in\rho(G/M)$. Without loss of generality, we can suppose that $p_1\in\rho(G/N)$. Since $\rho(G/N)$ is proper in $\rho(G)$, we then have three options as to what $\rho(G/N)$ can be: (i) $\{p_1,p_2,p_5,p_6,p_7\}$, (ii) $\{p_1,p_2,p_3,p_5,p_6,p_7\}$, or (iii) $\{p_1,p_2,p_4,p_5,p_6,p_7\}$. For (i), we have that $\rho(G/N)=\{p_1,p_2,p_5,p_6,p_7\}$, and one can see that no connected graph with this vertex set occurs as the prime character degree graph of any solvable group. The only possibility is for the disconnected graph to occur, with complete components $\{p_1,p_2\}$ and $\{p_5,p_6,p_7\}$. Then by Theorem 5.5 from \cite{L} we know that $G/N$ has a central Sylow $p_3$-subgroup or a central Sylow $p_4$-subgroup, in which case $O^{p_3}(G)<G$ or $O^{p_4}(G)<G$, respectively. This is our contradiction, since both $p_3$ and $p_4$ have been shown to be admissible, and so $O^{p_3}(G)=G$ and $O^{p_4}(G)=G$. For (ii), we have $\rho(G/N)=\{p_1,p_2,p_3,p_5,p_6,p_7\}$. However, in this scenario, no graph, connected or disconnected, can occur with that vertex set. This leaves us with (iii), giving $\rho(G/N)=\{p_1,p_2,p_4,p_5,p_6,p_7\}$. It is not hard to see that the only possible graph arising is the disconnected graph with complete components $\{p_1,p_2\}$ and $\{p_4,p_5,p_6,p_7\}$. Again by Theorem 5.5 from \cite{L}, this implies that $G/N$ has a central Sylow $p_3$-subgroup, which means that $G=O^{p_3}(G)<G$, a contradiction since $p_3$ is admissible.

Suppose (c) occurs; that is, suppose $\rho(G)\setminus(\rho(G/N)\cap\rho(G/M)) \subseteq \{p_2, p_5\}$. This implies that $\{p_1, p_3, p_4, p_6, p_7\} \subseteq \rho(G/N)\cap\rho(G/M)$. Since $\rho(G/N)$ and $\rho(G/M)$ are both proper in $\rho(G)$, and $\rho(G)=\rho(G/N)\cup\rho(G/M)$, then without loss of generality, we must have that $\rho(G/N)=\{p_1,p_2,p_3,p_4,p_6,p_7\}$ and $\rho(G/M)=\{p_1,p_3,p_4,p_5,p_6,p_7\}$. However, this is a contradiction because no graph can occur with the vertex set given to $\rho(G/N)$ as stipulated above, regardless if it is connected or disconnected.

Suppose (d) occurs; that is, suppose $\rho(G)\setminus(\rho(G/N)\cap\rho(G/M)) \subseteq \{p_4, p_5, p_6, p_7\}$. This implies that $\{p_1, p_2, p_3\} \subseteq \rho(G/N)\cap\rho(G/M)$. Similar to (a) above, since $p_7\in\rho(G)=\rho(G/N)\cup\rho(G/M)$, then we know that $p_7\in\rho(G/N)$ or $p_7\in\rho(G/M)$. Without loss of generality, we can suppose that $p_7\in\rho(G/N)$. Since $\rho(G/N)$ is proper in $\rho(G)$, we then have seven options as to what $\rho(G/N)$ can be: (i) $\{p_1,p_2,p_3,p_7\}$, (ii) $\{p_1,p_2,p_3,p_4,p_7\}$, (iii) $\{p_1,p_2,p_3,p_5,p_7\}$, (iv) $\{p_1,p_2,p_3,p_6,p_7\}$, (v) $\{p_1,p_2,p_3,p_4,p_5,p_7\}$, (vi) $\{p_1,p_2,p_3,p_4,p_6,p_7\}$, or (vii) $\{p_1,p_2,p_3,p_5,p_6,p_7\}$. For (i), we have that $\rho(G/N)=\{p_1,p_2,p_3,p_7\}$, and one can see that the only graph that arises here is the disconnected graph with complete components $\{p_1,p_2,p_3\}$ and $\{p_7\}$. Then using Theorem 5.5 from \cite{L}, we have that $G/N$ has a central Sylow $p_4$-subgroup, central Sylow $p_5$-subgroup, or central Sylow $p_6$-subgroup. Therefore, $O^{p_4}(G)<G$, $O^{p_5}(G)<G$, or $O^{p_6}(G)<G$, respectively. Regardless of the situation, we have a contradiction, as $O^{p_4}(G)=O^{p_5}(G)=O^{p_6}(G)=G$ since $p_4$, $p_5$, and $p_6$ are all admissible vertices. For (ii), (iii), and (iv), one can check that the only graph that can arise for $G/N$ is the disconnected graph with complete component $\{p_1,p_2,p_3\}$ and the other component having two vertices, one of which is $p_7$. In each case, we can again apply Theorem 5.5 from \cite{L} to conclude that $G/N$ has a central Sylow $p_i$-subgroup for some $4\leq i\leq6$, which is a contradiction since $p_i$ is admissible for each of these $i$. For (v), (vi), and (vii), notice that there is no graph, connected or disconnected, that can occur with each respective vertex set.

Hence, no such $N$ can occur, and we get our desired conclusion that the Fitting subgroup $F$ is minimal normal in $G$.
\end{proof}

\begin{proposition}\label{con1}
The graph $C_{18}$ does not occur as the prime character degree graph for any solvable group.
\end{proposition}
\begin{proof}
For the sake of contradiction, suppose that $C_{18}=\Delta(G)$ for some finite solvable group $G$, where $|G|$ is minimal. Observe that the Fitting subgroup $F$ of $G$ is minimal normal in $G$ by Lemma \ref{min1}. Next, consider the vertices $a=p_1$, $b=p_2$, $c=p_4$, and $d=p_5$, where we observe that $p_4$ and $p_5$ are admissible by Lemma \ref{ad1}. Applying Lemma \ref{thefinal} yields our contradiction, and therefore the graph $C_{18}$ is not the prime character degree graph for any solvable group.
\end{proof}

\subsection{Eliminating the graph $C_{20}$}\label{secC20}

Here our goal is to show that $C_{20}$ does not occur as the prime character degree graph for any solvable group (see Figure \ref{figC20}). Our strategy will mimic what was done in Section \ref{secC18}.

\begin{figure}[htb]
    \centering
$
\begin{tikzpicture}[scale=2.5]
\node (1) at (0,1) {$p_1$};
\node (2) at (0,0) {$p_2$};
\node (3) at (0.5,.5) {$p_3$};
\node (4) at (1.25, .75) {$p_4$};
\node (5) at (1.25, .25) {$p_5$};
\node (6) at (1.75, 1) {$p_6$};
\node (7) at (1.75, 0) {$p_7$};
\path[font=\small,>=angle 90]
(1) edge node [right] {$ $} (2)
(1) edge node [above] {$ $} (3)
(2) edge node [right] {$ $} (3)
(4) edge node [right] {$ $} (5)
(4) edge node [right] {$ $} (6)
(4) edge node [right] {$ $} (7)
(5) edge node [right] {$ $} (6)
(5) edge node [right] {$ $} (7)
(6) edge node [right] {$ $} (7)
(1) edge node [right] {$ $} (6)
(2) edge node [right] {$ $} (7)
(3) edge node [right] {$ $} (4)
(3) edge node [right] {$ $} (5);
\end{tikzpicture}
$
    \caption{The graph $C_{20}$}
    \label{figC20}
\end{figure}
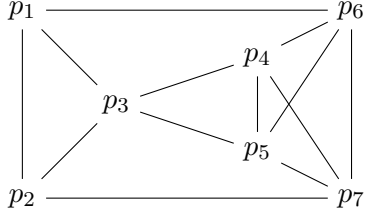

\begin{lemma}\label{ad2}
Assume $C_{20}=\Delta(G)$ for some finite solvable group $G$, where $|G|$ is minimal. Then $G$ does not have a normal nonabelian Sylow $p_i$-subgroup for $i\in\{3,4,5,6,7\}$. In particular, each $p_i$ is a strongly admissible vertex.
\end{lemma}
\begin{proof}
We begin with the vertex $p_3$. Removing $p_3$ and all incident edges leaves us with the graph $C_{20}[p_3] = \Gamma_{4,2}$, which is non-occurring by \cite{BL} (see Theorem \ref{KT}). We cannot lose either of the edges $\epsilon(p_1, p_3)$ or $\epsilon(p_2, p_3)$, for doing so would violate P\'alfy's condition. Therefore, neither of these edges can be lost. Vertices $p_4$ and $p_5$ are symmetric, thus making edges $\epsilon(p_3, p_4)$ and $\epsilon(p_3, p_5)$ symmetric as well. Removing either of those edges would create the non-occurring graph $\Gamma_{4,3}$, and removing both edges would create the graph $C_{20}[\epsilon(p_3, p_4), \epsilon(p_3, p_5)] = C_4$. Neither of these graphs occur by Theorem \ref{KT} and Section \ref{secCDT}, respectively. Thus, $p_3$ is admissible.

For strong admissibility, we consider the removal of $p_3$, all incident edges, and one or more of the edges between two vertices adjacent to $p_3$. There are two edges in consideration: $\epsilon(p_4,p_5)$ and $\epsilon(p_1,p_2)$. If we lose the edge $\epsilon(p_4, p_5)$, our new graph violates P\'alfy's condition and hence that edge cannot be lost. Next, removing the edge $\epsilon(p_1, p_2)$ would create the graph $C_{20}[p_3, \epsilon(p_1, p_2)]$, and this graph does not occur by \cite{BLL}. Hence, $p_3$ is strongly admissible.

Next, we consider the vertex $p_4$. Removing $p_4$ gives us the graph $C_{20}[p_4] = \Gamma_{3,3}$, which is non-occurring by \cite{BL} (again, see Theorem \ref{KT}). We cannot remove any combination of the edges $\epsilon(p_4, p_5)$, $\epsilon(p_4, p_6)$, and $\epsilon(p_4, p_7)$, as doing so would violate P\'alfy's condition. This leaves us to check the edge $\epsilon(p_3, p_4)$, but we established previously that losing this edge would create the graph $\Gamma_{4,3}$. Therefore, $p_4$ is admissible.

For strong admissibility, we consider the removal of $p_4$, all incident edges, and one or more of the edges between two vertices adjacent to $p_4$. It is easy to see that the removal of one or more of the edges $\epsilon(p_5, p_6)$, $\epsilon(p_5, p_7)$, and $\epsilon(p_6, p_7)$ would violate P\'alfy's condition, and therefore cannot be lost. The only other edge to consider is between $p_3$ and $p_5$. However, the graph $C_{20}[p_4, \epsilon(p_3, p_5)]$ does not occur by \cite{BLL}. Thus, $p_4$ is strongly admissible. By symmetry, we also have that $p_5$ is strongly admissible.

We now check for admissibility of $p_6$. Removal of the vertex $p_6$ gives us the graph $C_{20}[p_6] = \Sigma^L_{2,1}$, which is non-occurring by \cite{LM} (or see \cite{BLL}). We cannot lose one or more of the edges between $p_6$ and any of the vertices $p_4$, $p_5$, and $p_7$, as doing so would create graphs that violate P\'alfy's condition. If we remove the edge $\epsilon(p_1, p_6)$, we are left with the graph $C_6$ which has been shown to not occur in Section \ref{secCDT}. Hence, $p_6$ is admissible.

For strong admissibility, we consider the removal of $p_6$, all incident edges, and one or more of the edges between two vertices adjacent to $p_6$. Similar to the case for $p_4$, it is easy to see that removing one or more of the edges $\epsilon(p_4, p_5)$, $\epsilon(p_4, p_7)$, or $\epsilon(p_5, p_7)$ from $C_{20}[p_6]$ would result in graphs that violate P\'alfy's condition. Therefore, $p_6$ is strongly admissible. Again by symmetry, we also have that $p_7$ is strongly admissible.

Since $p_i$ is strongly admissible for each $i \in \{3,4,5,6,7\}$, then there is no normal nonabelian Sylow $p_i$-subgroup for each $i$ by way of \cite{BL2} (see Lemma \ref{strong}).
\end{proof}

\begin{lemma}\label{pi2}
Assume $C_{20}=\Delta(G)$ for some finite solvable group $G$, where $|G|$ is minimal. There exists no normal nonabelian Sylow $p_j$-subgroup for $j\in\{1,2\}$.
\end{lemma}
\begin{proof}
First, we will begin by considering the vertex $p_1$. We will follow the notation from Lemma \ref{pi}. For $q=p_1$, we have $\pi=\{p_2, p_3, p_6\}$ and $\rho= \{p_4, p_5, p_7\}$. Setting $\pi_1 = \{p_6\}$ and $\pi_2 = \{p_2, p_3\}$, we choose $v = p_3$, which is adjacent to $s = p_4\in\rho$, shown to be admissible in Lemma \ref{ad2}. We also note that $w = p_7 \in \rho$ is not adjacent to $v$, and hence the result follows from Lemma \ref{pi}.

By symmetry, we can conclude that there is no normal nonabelian Sylow $p_2$-subgroup.
\end{proof}

\begin{lemma}\label{sub2}
No proper subgraph with the same vertex set as $C_{20}$ occurs as the prime character degree graph of any solvable group.
\end{lemma}
\begin{proof}
We have established previously that we cannot lose any edges in the complete graphs on $\{p_1, p_2, p_3\}$ and $\{p_4, p_5, p_6, p_7\}$. Thus, we are left with four edges to consider: $\epsilon(p_1, p_6)$, $\epsilon(p_2, p_7)$, $\epsilon(p_3, p_4)$, and $\epsilon(p_3, p_5)$. We will check all subgraphs for which we remove one or more of those edges.

First, recall that $p_1$ and $p_2$ are symmetric, and $p_6$ and $p_7$ are also symmetric. Thus, it follows that the edges $\epsilon(p_1, p_6)$ and $\epsilon(p_2, p_7)$ are symmetric. Similarly, recall that $p_4$ and $p_5$ are symmetric. This means that $\epsilon(p_3, p_4)$ and $\epsilon(p_3, p_5)$ are symmetric as well. Notice that $C_{20}[\epsilon(p_1, p_6)] = C_{20}[\epsilon(p_2, p_7)] = C_{6}$ and $C_{20}[\epsilon(p_3, p_4)] = C_{20}[\epsilon(p_3, p_5)] = \Gamma_{4,3}$ have both been shown not to occur.

Removing both edges $\epsilon(p_1, p_6)$ and $\epsilon(p_2, p_7)$ gives us the graph $C_{20}[\epsilon(p_1, p_6), \epsilon(p_2, p_7)] = C_{2}$ which does not occur (see Section \ref{secCDT}). In a similar fashion, losing both edges $\epsilon(p_3, p_4)$ and $\epsilon(p_3, p_5)$ gives us the graph $C_{20}[\epsilon(p_3, p_4), \epsilon(p_3, p_5)] = C_{4}$, which again does not occur. Without loss of generality (due to symmetry), we have that $C_{20}[\epsilon(p_1, p_6), \epsilon(p_3, p_4)] = C_{4}$ as well.  

If we remove any combination of three of our four viable edges, we are left with the graph $C_{1}$, which has been shown not to occur in Section \ref{secCDT}. Finally, removing all four edges will give us the disconnected graph in Figure \ref{figDDNO}, which was also shown not to occur.

Hence, no proper subgraph of $C_{20}$ with the same vertex set occurs as the prime character degree graph of any solvable group.
\end{proof}

Just as before, assuming that there exists some solvable group $G$ with $|G|$ minimal such that $\Delta(G)=C_{20}$, we again have the conclusion that there are no normal nonabelian Sylow $p$-subgroups for all $p\in\rho(G)$ (using Lemmas \ref{ad2} and \ref{pi2}). With $F$ as the Fitting subgroup of $G$, we then note that $\rho(G)=\pi(|G:F|)$. Therefore, $\rho(G)=\rho(G/\Phi(G))$. Lemma \ref{sub2} verifies there is no proper subgraph of $C_{20}$ with the same vertex set which occurs as the prime character degree graph of any solvable group, and therefore, as $|G|$ is minimal, we can conclude that $\Phi(G)$ is trivial. We can then apply Lemma III 4.4 from \cite{H} to obtain the existence of a subgroup $H$ of $G$ such that $G=HF$ and $H\cap F=1$.

\begin{lemma}\label{min2}
Assume $C_{20}=\Delta(G)$ for some finite solvable group $G$, where $|G|$ is minimal. Then the Fitting subgroup $F$ of $G$ is minimal normal in $G$.
\end{lemma}
\begin{proof}
We follow an argument almost identical to that which was done in Lemma \ref{min1}. As before, we can follow the assumptions and then can define $H$ as above. Also, set $E$ to be the Fitting subgroup of $H$.

Next, we proceed by contradiction. Thus, we suppose that there exists some normal subgroup $N$ of $G$ such that $1<N<F$. We aim to show that no such $N$ can exist.

Notice once more by \cite{H}, particularly III 4.5, that we know there exists some normal subgroup $M$ of $G$ such that $F=N\times M$, and since both $N$ and $M$ are nontrivial, we have that $\rho(G/N)\subset\rho(G)$ and $\rho(G/M)\subset\rho(G)$. For any prime $q\in\rho(G)\setminus\rho(G/N)$, it is known that $G/N$ must have a normal nonabelian Sylow $q$-subgroup whose class is a formation. Thus, it is required that $q\in\rho(G/M)$. Hence, we get that $\rho(G)=\rho(G/N)\cup\rho(G/M)$.

So for this $q\in\rho(G)\setminus\rho(G/N)$, notice that $q\not\in\rho(G/F)=\rho(H)$, and so $E$ must contain the Sylow $q$-subgroup of $H$. Since $q\in\rho(G)$, however, we must have that $q$ divides $|H|$, and so $q$ divides $|E|$ too. Then since $\cd(G)$ contains a degree divisible by all the prime divisors of $|EF:F|=|E|$, we get that $\rho(G)\setminus(\rho(G/N)\cap\rho(G/M))$ lies in a complete subgraph of $\Delta(G)$. Hence, $\rho(G)\setminus(\rho(G/N)\cap\rho(G/M))$ must lie in one of the following sets: (a) $\{p_1, p_2, p_3\}$, (b) $\{p_1, p_6\}$, (c) $\{p_2, p_7\}$, (d) $\{p_3, p_4, p_5\}$, or (e) $\{p_4, p_5, p_6, p_7\}$.

Suppose (a) occurs; that is, suppose $\rho(G)\setminus(\rho(G/N)\cap\rho(G/M)) \subseteq \{p_1, p_2, p_3\}$. Following \cite{L3}, we see that $E$ has a Hall $\{p_1,p_2,p_3\}$-subgroup of $H$. One can then conclude that $E$ is in fact the Hall $\{p_1,p_2,p_3\}$-subgroup since $\cd(G)$ has a degree that is divisible by all the primes dividing $|E|$ and $|E|$ is divisible by no other primes. Next, we can find a character $\chi\in\Irr(G)$ with $p_1p_2p_3$ dividing $\chi(1)$. Letting $\theta$ be an irreducible constituent of $\chi_{FE}$, one can see that $\chi(1)/\theta(1)$ divides $|G:FE|$ and that $\chi(1)$ is relatively prime to $|G:FE|$. This forces $\chi_{FE}=\theta$. Since $p_1$, $p_2$, and $p_3$ all divide $\theta(1)$, and the only possible primes that could divide a character in $\cd(G/FE)$ are $p_4$, $p_5$, $p_6$, or $p_7$, we get that $\cd(G/FE)=\{1\}$ by way of Gallagher's Theorem. Hence $G/FE$ is abelian. For instance, this then yields $O^{p_4}(G)<G$, a contradiction as $p_4$ is admissible.

Suppose (b) occurs; that is, suppose $\rho(G)\setminus(\rho(G/N)\cap\rho(G/M)) \subseteq \{p_1, p_6\}$. This implies that $\{p_2, p_3, p_4, p_5, p_7\} \subseteq \rho(G/N)\cap\rho(G/M)$. Since $\rho(G/N)$ and $\rho(G/M)$ are both proper in $\rho(G)$, and $\rho(G)=\rho(G/N)\cup\rho(G/M)$, then without loss of generality, we must have that $\rho(G/N)=\{p_1,p_2,p_3,p_4,p_5,p_7\}$ and $\rho(G/M)=\{p_2,p_3,p_4,p_5,p_6,p_7\}$. However, this is a contradiction because no graph can occur with the vertex set given to $\rho(G/N)$ as defined above, regardless if it is connected or disconnected. By symmetry, no graph can occur when $\rho(G)\setminus(\rho(G/N)\cap\rho(G/M)) \subseteq \{p_2, p_7\}$, thus eliminating case (c).

Suppose (d) occurs; that is, suppose $\rho(G)\setminus(\rho(G/N)\cap\rho(G/M)) \subseteq \{p_3, p_4, p_5\}$. This implies that $\{p_1, p_2, p_6, p_7\} \subseteq \rho(G/N)\cap\rho(G/M)$. Since $p_3\in\rho(G)=\rho(G/N)\cup\rho(G/M)$, then we know that $p_3\in\rho(G/N)$ or $p_3\in\rho(G/M)$. Without loss of generality, we can suppose that $p_3\in\rho(G/N)$. Since $\rho(G/N)$ is proper in $\rho(G)$, we then have three options as to what $\rho(G/N)$ can be: (i) $\{p_1,p_2,p_3,p_6,p_7\}$, (ii) $\{p_1,p_2,p_3,p_4,p_6,p_7\}$, or (iii) $\{p_1,p_2,p_3,p_5,p_6,p_7\}$. For (i), we have that $\rho(G/N)=\{p_1,p_2,p_3,p_6,p_7\}$, and one can see that no connected graph with this vertex set occurs as the prime character degree graph of any solvable group. The only possibility is for the disconnected graph to occur, with complete components $\{p_1,p_2,p_3\}$ and $\{p_6,p_7\}$. Then by Theorem 5.5 from \cite{L} we know that $G/N$ has a central Sylow $p_4$-subgroup or a central Sylow $p_5$-subgroup, in which case $O^{p_4}(G)<G$ or $O^{p_5}(G)<G$, respectively. This is our contradiction, since both $p_4$ and $p_5$ have been shown to be admissible in Lemma \ref{ad2}, and so $O^{p_4}(G)=G$ and $O^{p_5}(G)=G$. For (ii), we have $\rho(G/N)=\{p_1,p_2,p_3,p_4,p_6,p_7\}$. No graph can occur with the vertex set given to $\rho(G/N)$ as defined above, regardless if it is connected or disconnected. Recall that the vertices $p_4$ and $p_5$ are symmetric. Then for (iii), this argument follows identically to (ii).

Suppose (e) occurs; that is, suppose $\rho(G)\setminus(\rho(G/N)\cap\rho(G/M)) \subseteq \{p_4, p_5, p_6, p_7\}$. This implies that $\{p_1, p_2, p_3\} \subseteq \rho(G/N)\cap\rho(G/M)$. Similar to (d) above, since $p_7\in\rho(G)=\rho(G/N)\cup\rho(G/M)$, then we know that $p_7\in\rho(G/N)$ or $p_7\in\rho(G/M)$. Without loss of generality, we can suppose that $p_7\in\rho(G/N)$. Since $\rho(G/N)$ is proper in $\rho(G)$, we then have seven options as to what $\rho(G/N)$ can be: (i) $\{p_1,p_2,p_3,p_7\}$, (ii) $\{p_1,p_2,p_3,p_4,p_7\}$, (iii) $\{p_1,p_2,p_3,p_5,p_7\}$, (iv) $\{p_1,p_2,p_3,p_6,p_7\}$, (v) $\{p_1,p_2,p_3,p_4,p_5,p_7\}$, (vi) $\{p_1,p_2,p_3,p_4,p_6,p_7\}$, or (vii) $\{p_1,p_2,p_3,p_5,p_6,p_7\}$. For (i), we have that $\rho(G/N)=\{p_1,p_2,p_3,p_7\}$, in which case this forces $\rho(G/M)=\{p_1,p_2,p_3,p_4,p_5,p_6\}$. Notice that neither the connected nor disconnected graph with vertex set $\rho(G/M)$ occurs as the prime character degree graph of any solvable group. For (ii), notice that $\rho(G/N)=\{p_1,p_2,p_3,p_4,p_7\}$, and the only graph that can arise is the disconnected graph with complete components $\{p_1,p_2,p_3\}$ and $\{p_4,p_7\}$. Applying Theorem 5.5 from \cite{L}, we then get that $G/N$ has a central Sylow $p_5$-subgroup or a central Sylow $p_6$-subgroup. This implies that $O^{p_5}(G)<G$ or $O^{p_6}(G)<G$, which is a contradiction since both $p_5$ and $p_6$ are admissible, in which case $O^{p_5}(G)=O^{p_6}(G)=G$. For (iii), the argument follows (ii) by symmetry. For (iv), $\rho(G/N)=\{p_1, p_2, p_3, p_6, p_7\}$. Once again, the only graph that can occur is the disconnected graph with complete components $\{p_1,p_2,p_3\}$ and $\{p_6,p_7\}$. We then apply Theorem 5.5 from \cite{L} and get that $G/N$ has a central Sylow $p_4$-subgroup or a central Sylow $p_5$-subgroup. This gives $O^{p_4}(G)<G$ or $O^{p_5}(G)<G$, which is a contradiction since both $p_4$ and $p_5$ are admissible, meaning that $O^{p_4}(G)=O^{p_5}(G)=G$. For (v), (vi), and (vii), notice that there is no graph, connected or disconnected, that can occur with each respective vertex set.

Hence, no such $N$ can occur, and we get our desired conclusion that the Fitting subgroup $F$ is minimal normal in $G$.
\end{proof}

\begin{proposition}\label{con2}
The graph $C_{20}$ does not occur as the prime character degree graph for any solvable group.
\end{proposition}
\begin{proof}
For the sake of contradiction, suppose that $C_{20}=\Delta(G)$ for some finite solvable group $G$, where $|G|$ is minimal. Notice that the Fitting subgroup $F$ of $G$ is minimal normal in $G$ by Lemma \ref{min2}. Next, take the vertices $a=p_1$, $b=p_2$, $c=p_4$, and $d=p_5$, where we observe that $p_4$ and $p_5$ are admissible by Lemma \ref{ad2}. Employing Lemma \ref{thefinal} grants our contradiction, and therefore the graph $C_{20}$ is not the prime character degree graph for any solvable group, as desired.
\end{proof}

\subsection{Unclassified graphs}\label{secCmaybe}

As mentioned, there are thirty-three graphs here that are unclassified (see Figures \ref{figC19} and \ref{figCM}). Setting
$$
\mathcal{J}:=\{8,13,15,21,22,23,24,25,x,52~|~27\leq x\leq49\},
$$
we then have the graphs $C_{19}$ (Figure \ref{figC19}) and $C_j$ with $j\in\mathcal{J}$ (Figure \ref{figCM}).

\begin{figure}[htb]
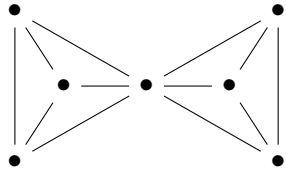

    \centering
$

$
    \caption{The graph $C_{19}$}
    \label{figC19}
\end{figure}

The graph $C_{19}$ is worth looking at on its own because, in general, it has a structure that is eerily similar to the bowtie graph that Lewis constructs in \cite{L3}. One wonders if $C_{19}$ can be constructed in a similar way.

\begin{figure}[b]
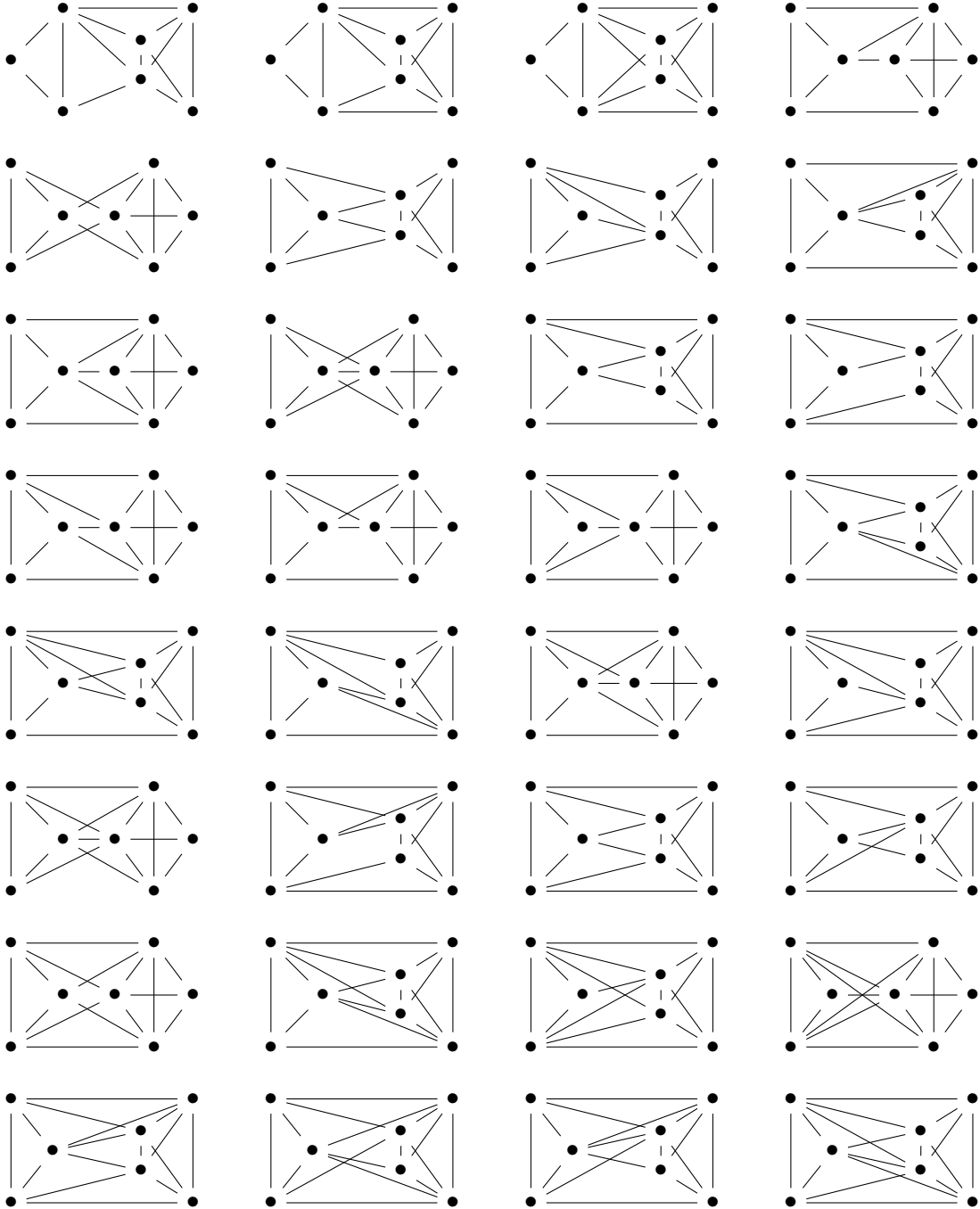

    \centering
$
%
$
    \caption{The graphs $C_j$ for $j\in\mathcal{J}$}
    \label{figCM}
\end{figure}

One of the main reasons these graphs went unclassified is that they have a subgraph with six vertices that occurs. For example, the graph $C_8$ has the lone diameter three graph constructed in \cite{L2} as a subgraph. Having subgraphs that occur often leads to issues with our method because we rely on vertices being admissible. This problem then compounds rapidly, because once we know $C_8$ cannot be tamed, then any graph that has $C_8$ as a subgraph can also not be tamed (like $C_{13}$ for instance). More work is constantly being done in this area, so one could hope that more graphs can be classified in the future. Papers like \cite{SS} however only cooperate with graphs that have an even number of vertices, so their results could not be applied to our remaining graphs. On the other hand, a paper like \cite{Z} does relate to some of our remaining graphs, but could not be applied based on our current approach.

We have also seen that many of these graphs that went unclassified rely upon the nine graphs with six vertices that were not classified in \cite{BLL}. In particular, we found that one of these graphs with six vertices, denoted $\Sigma_{2,1}^R$ from \cite{DLM}, has far-reaching consequences. Hypothetically, if one could conclude that $\Sigma_{2,1}^R$ does not occur, then this will imply that $C_{21}$, $C_{22}$, and $C_{32}$ all do not occur. One of the issues with this graph $\Sigma_{2,1}^R$ is that it does not quite satisfy the technical hypothesis from \cite{BLL2}; further methods will need to be explored.

\appendix

\section{Connected Graphs with Cliques of One and Six}
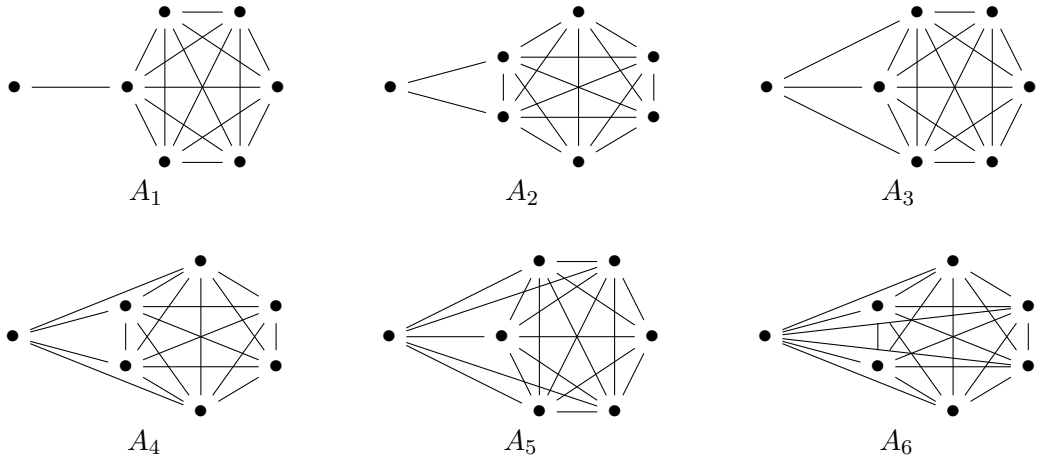
\label{AA}

$$

$$

\section{Connected Graphs with Cliques of Two and Five}
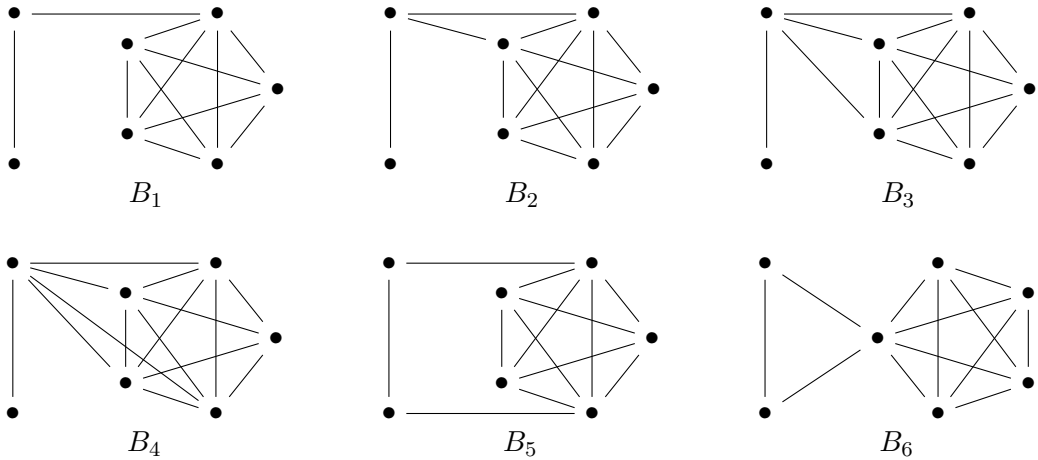
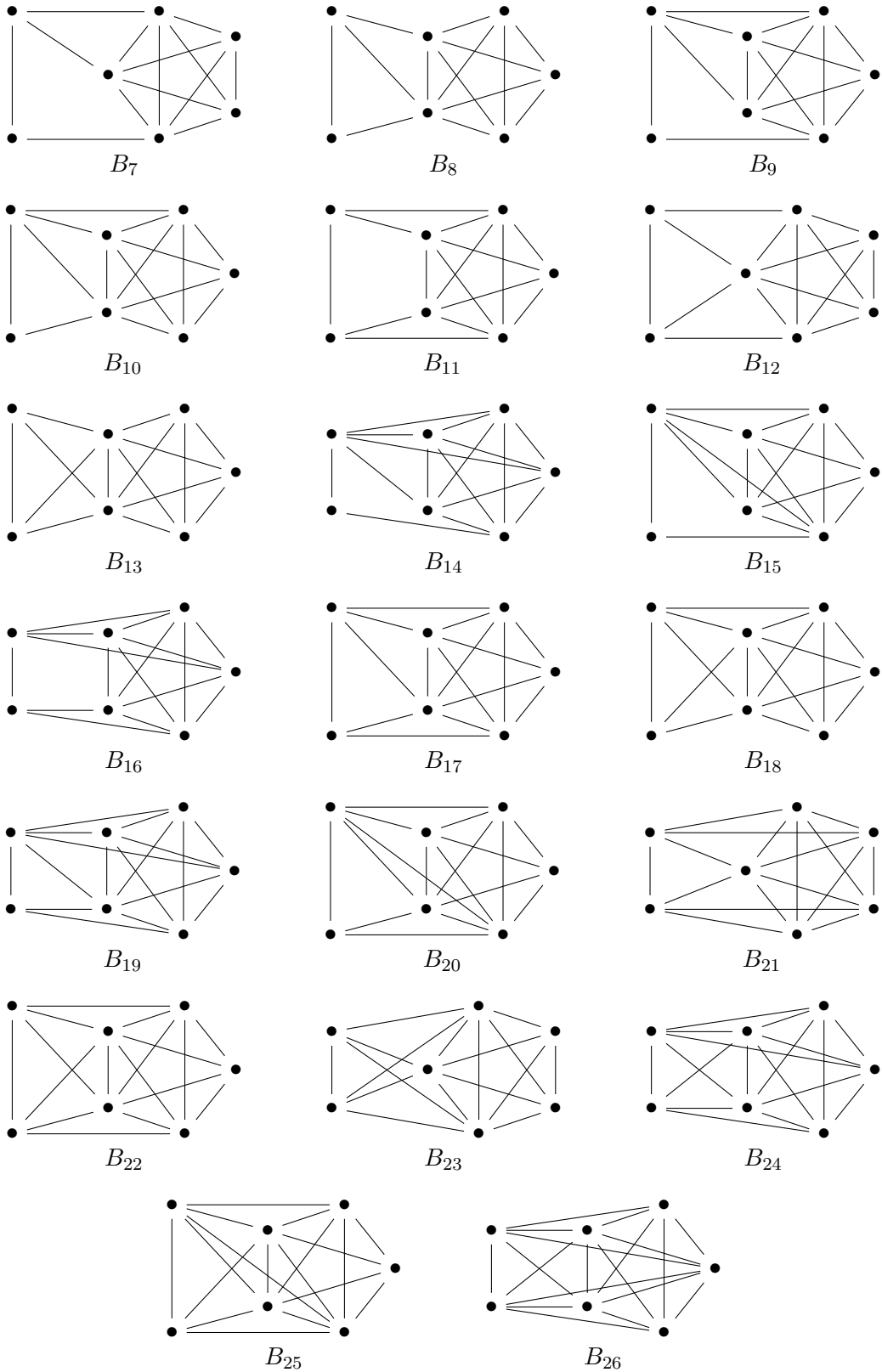
\label{AB}

$$
%
$$

\section{Connected Graphs with Cliques of Three and Four}\label{AC}

$$
%
$$

\end{document}